\documentclass[11pt,a4paper]{article}
\usepackage{indentfirst,latexsym,bm}
\usepackage{amscd}
\usepackage{amsthm}
\usepackage{amsmath}
\usepackage{amsfonts}
\usepackage{amssymb}
\usepackage[all]{xy}
\usepackage[dvips,draft]{graphics}
\usepackage{hyperref}
\usepackage{url}
\allowdisplaybreaks

\begin{document}
\title{\bf Chamber structure for some equivariant relative Gromov-Witten invariants of $\mathbb{P}^1$ in genus $0$}
\author{\sl Longting Wu\footnote{Research of the author was partially supported by SRFDP grant 20120001110051.}}
\date{}
\maketitle

\begin{abstract}
In this paper, we study genus $0$ equivariant relative Gromov-Witten invariants of $\mathbb{P}^1$ whose corresponding relative stable maps are totally ramified over one point. For fixed number of marked points, we will show that such invariants are piecewise polynomials in some parameter space. The parameter space can then be divided into polynomial domains, called chambers. We determine the difference of polynomials between two neighbouring chambers. In some special chamber, which we called the totally negative chamber, we show that such a polynomial has a simple expression. The chamber structure here shares some similarities to that of double Hurwitz numbers (Goulden et al., 2005)\cite{GJV} (Shadrin et al., 2008)\cite{SSV} (Cavalieri et al., 2011)\cite{CJM} (Johnson, 2015)\cite{J}.

\end{abstract}

\numberwithin{equation}{section}

\newtheorem{theorem}{Theorem}[section]
\newtheorem{conjecture}[theorem]{Conjecture}
\newtheorem{lemma}[theorem]{Lemma}
\newtheorem{corollary}[theorem]{Corollary}
\theoremstyle{remark}
\newtheorem{remark}{Remark}
\theoremstyle{definition}
\newtheorem{example}{Example}
\newtheorem{definition}[theorem]{Definition}

\section{Introduction}
In \cite{CJM}\cite{GJV}\cite{J}\cite{SSV}, the chamber structure for double Hurwitz numbers has been determined. In this paper, we will show that a similar chamber structure also appears for the following genus $0$ equivariant relative Gromov-Witten invariants of $\mathbb{P}^1$:
\begin{equation}\label{key}
\Big\langle d\Big |\prod_{i=1}^m \prod_{k_i=0}^{l_i}(k_i\tau+[\infty])\Big\rangle_0^{T}
\end{equation}
see (\ref{keyinv}) for precise definition.

Without further explanation, we always assume that $m\geq 2$. By the dimension constraint (\ref{virdim}), it is easy to see that if $1+\sum_{i=1}^m l_i< d$, then (\ref{key}) vanishes. So we may assume that $1+\sum_{i=1}^m l_i\geq d$.

Let $t$ be the generator of $H^*_{\mathbb{C}^*}(pt)$ corresponding to the dual of the standard representation of $\mathbb{C}^*$. In other words, $t$ is the hyperplane class of $\mathbb{C}\mathbb{P}^{\infty}$. Then we may write (\ref{key}) as
\[\Big\langle d\Big |\prod_{i=1}^m \prod_{k_i=0}^{l_i}(k_i\tau+[\infty])\Big\rangle_0^{T}=F(l_1,\ldots,l_m,d)t^{1+\sum l_i-d}\]
where $F(l_1,\ldots,l_m,d)\in\mathbb{Q}$.

In order to formulate our results, we need to introduce some notation. We treat $x_1=l_1,\ldots,x_m=l_m,y=d$ as coordinates of $\mathbb{R}^{m+1}$. The parameter space is given by
\[S_{m}:=\bigr\{(x_1,\ldots,x_m,y)\in \mathbb{R}^{m+1}\bigr|x_1\geq x_2\ldots \geq x_m\geq 0,~y\geq 1,~1+\sum_{i=1}^mx_i\geq y\bigr\}.\]
A resonance is a hyperplane
\[H_{I}:=\bigr\{(x_1,\ldots,x_m,y)\in S_m~\bigr|~\sum_{i\in I}x_i=y\bigr\}\]
where $I\subset \{1,2,\ldots,m\}$ and $|I|\leq m-2$. The connected components of the complement to the union of all resonances are called chambers.

Now our results can be summarized as follows.
\begin{theorem}\label{keythm1}
If $(x_1,\ldots,x_m,y)\in S_m\cap\mathbb{Z}^{m+1}$ varies in the closure of one chamber, then $F(x_1,\ldots,x_m,y)$ can be expressed as a polynomial of $x_1,\ldots,x_m,y$ with degree at most $2m-4$.
\end{theorem}

Let $\mathfrak{c}$ be a chamber. Then we use $P_{\mathfrak{c}}$ to denote the corresponding polynomial $F$. When chamber $\mathfrak{c}$ changes, $P_{\mathfrak{c}}$ may be different. The difference can be described by a wall crossing formula.
\begin{definition}
Let $\mathfrak{c}_1$ and $\mathfrak{c}_2$ be two neighbouring chambers which are bounded by the wall $H_I$. Suppose that $\sum_{i\in I} x_i<y$ in $\mathfrak{c}_1$. We define
\[WC_{I}(x_1,\ldots,x_m,y):=P_{\mathfrak{c}_1}(x_1,\ldots,x_m,y)-P_{\mathfrak{c}_2}(x_1,\ldots,x_m,y).\]
\end{definition}
Now the wall crossing formula is just a polynomial formula for $WC_{I}$. In order to formulate our wall crossing formula, we need to introduce some notation.

\begin{definition}\label{par2}
Choosing a nonempty subset $J\subset \{1,2,\ldots,m\}$. We define $P_J^{\geq 2}$ to be the set of those partitions of $J$ such that each subset contains at least two elements.
\end{definition}

We set
\begin{equation}\label{lamW}
\mathcal{W}(q):=\sum_{n=1}^{\infty}\frac{n^{n-1}}{n!}q^n
\end{equation}
to be a power series of $q$ whose radius of convergence is $e^{-1}$. For $z\in \mathbb{R}$, we define the operator
\begin{equation}\label{oper}
\mathcal{O}_z:=\frac{\mathcal{W}}{1-\mathcal{W}}\frac{d}{d\mathcal{W}}+z.
\end{equation}
Let $w$ be an indeterminate and $a\in \mathbb{Z}_{\geq 0}$. We define
\[(w+1)_a:=
\begin{cases}
  (w+1)(w+2)\ldots(w+a), & \mbox{if } a>0, \\
  1, & \mbox{if } a=0.
\end{cases}
\]

Let $w_1,w_2,\ldots,w_m$ be indeterminates. For each $I\subset \{1,2,\ldots,m\}$, we define
\begin{equation}\label{R_I}
\begin{aligned}
R_I:=&\Bigg[\prod_{j\in I^c}w_j^{x_j}q^{y-\sum_{i\in I}(x_i+1)}\Bigg]\prod_{j\in I^c}(w_j+1)_{x_j}\sum_{\{\mathcal{I}_1,\ldots \mathcal{I}_s\}\in P_{I^c}^{\geq 2}}(-1)^{s+|I|-1}y^{s+|I|-2}\\
&\times\frac{e^{-\bigr(y-\sum_{i\in I}(x_i+1)\bigr)\mathcal{W}}}{(1-\mathcal{W})^{|I|}}\prod_{r=1}^s\mathcal{O}_{
\sum_{j\in \mathcal{I}_r}w_j}^{|\mathcal{I}_r|-2}\Bigg(\frac{\mathcal{W}}{(1-\mathcal{W})^3}\Bigg).
\end{aligned}
\end{equation}
Here, $I^c$ is the complement of $I$ in $\{1,2,\ldots,m\}$, $(x_1,\ldots,x_m,y)\in S_m\cap \mathbb{Z}^{m+1}$ and $\mathcal{O}_{
\sum_{j\in \mathcal{I}_r}w_j}^{|\mathcal{I}_r|-2}$ is an operator given by the $(|\mathcal{I}_r|-2)$-th power of $\mathcal{O}_{
\sum_{j\in \mathcal{I}_r}w_j}$. We always use the notation $[A]B$ to denote the coefficient of term $A$ in $B$.

We remark that if $I=\{1,2,\ldots,m\}$, then $I^c$ is the empty set. In this case, $R_I$ should be understood as
\begin{equation}\label{s-R_I}
\Big[q^{y-\sum_{i=1}^m(x_i+1)}\Big](-1)^{m-1}y^{m-2}\frac{e^{-\bigr(y-\sum_{i=1}^m(x_i+1)\bigr)\mathcal{W}}}{(1-\mathcal{W})^{m}}.
\end{equation}

By Lemma \ref{keylem}, we know that if $|I|\leq m-2$ and $\sum_{i\in I}x_i\leq y$, then $R_I$ can be expressed as a polynomial $\mathcal{R}_I$ in terms of $x_1,x_2,\ldots,x_m,y$ with degree at most $2m-4$. The proof of Lemma \ref{keylem} itself also gives an effective way to compute $\mathcal{R}_I$.

Now our wall crossing formula becomes
\begin{theorem}\label{keythm2}
\[WC_I=\mathcal{R}_I.\]
\end{theorem}

Our next theorem shows that there exists one special chamber such that $P_{\mathfrak{c}}$ has a simple expression.

\begin{definition}
The totally negative chamber $\mathfrak{c}^{tn}$ is a chamber such that $\sum_{i\in I}x_i<y$ for all $I\subset \{1,2,\ldots,m\}$ which satisfy $|I|\leq m-2$.
\end{definition}

For totally negative chamber $\mathfrak{c}^{tn}$, we have
\begin{theorem}\label{keythm3}
\begin{equation}\label{tnfm}
P_{\mathfrak{c}^{tn}}=y^{m-2}{1+\sum_{i=1}^m x_i-y+m-2\choose m-2}
\end{equation}
where ${*\choose *}$ is the binomial coefficient.
\end{theorem}

Obviously, $P_{\mathfrak{c}^{tn}}$ is a degree $2m-4$ polynomial in terms of $x_1,\ldots,x_m,y$. We remark that by using the convention of (\ref{conven}), the R.H.S of (\ref{tnfm}) also gives $F$ when $m=1$. This can be easily deduced by using relative virtual localization formula (\ref{rvlf}). We will omit the details of computation here.

By Theorems \ref{keythm2}, \ref{keythm3}, each $P_{\mathfrak{c}}$ can be determined by choosing a path from the given chamber $\mathfrak{c}$ to the totally negative chamber $\mathfrak{c}^{tn}$ and summing over all the differences.

Observe that those points $(x_1,\ldots,x_m,y)\in S_m$ which satisfy
\[1+\sum_{i=1}^m x_i-y=0\]
sit on the totally negative chamber $\mathfrak{c}^{tn}$. So we can use (\ref{tnfm}) to compute them. In that case, if we take the non-equivariant limit of (\ref{key}), then we get
\begin{corollary}\label{keycor}
If $1+\sum_{i=1}^m l_i=d$, then we have
\[\Big\langle d\Big|\prod_{i=1}^m\tau_{l_i}([pt])\Big\rangle_{0,d}=\frac{d^{m-2}}{l_1!\ldots l_m!}.\]
Here, the L.H.S in the above identity is defined by (\ref{neqin}).
\end{corollary}

We remark that Corollary \ref{keycor} is a generalization of Lemma 1.4 in \cite{OP1} which gives the case when $m=1$.

The motivation to consider equivariant relative Gromov-Witten invariants of $\mathbb{P}^1$ in such a form comes from the author's work \cite{W}. In that paper, we need to compute some relative invariants of $\mathbb{P}^1$-bundle, by \cite{MP} Section 1.2, which is equivalent to compute equivariant relative Gromov-Witten invariants of $\mathbb{P}^1$ in the form of (\ref{key}).

Although (\ref{key}) can be computed by using the infinite wedge formalism in \cite{OP3}, it seems that it is not easy to directly deduce the chamber structure of (\ref{key}). We will compute (\ref{key}) by using relative virtual localization formula (\ref{rvlf}). The assignment of absolute marked points in the localization graphs will naturally give the chamber structure. We also use some combinatorial tricks to sum over all the contributions in the relative virtual localization formula. In the proofs of Theorems \ref{keythm1}, \ref{keythm2}, \ref{keythm3}, we will always assume that $m\geq 2$.

The paper is organized as follows. Section \ref{1} gives a brief introduction of relative Gromov-Witten invariants and relative virtual localization formula. Section \ref{2} gives proofs of Theorems \ref{keythm1}, \ref{keythm2}, \ref{keythm3}. Section \ref{3} gives examples of chamber structure when $m=2,3$.

\textbf{Acknowledgement.}The author thanks Professor Xiaobo Liu for his patience and guidance during all the time.

\section{Preliminaries}\label{1}
\subsection{Relative Gromov-Witten invariants of $\mathbb{P}^1$}\label{pre1}
The relative Gromov-Witten invariants were defined by Ionel-Parker \cite{IP} and Li-Ruan \cite{LR} in symplectic geometry, and Jun Li \cite{Li} in algebraic geometry.

Let $q_0=[1,0]$ and $q_{\infty}=[0,1]$ be two points of $\mathbb{P}^1$. We will introduce the relative Gromov-Witten invariants of the pair $(\mathbb{P}^1,q_0)$ using Jun Li's algebraic version.

\begin{definition}
Let $\mathbb{P}^1(m)$ be the union of $m$ copies of $\mathbb{P}^1$, such that the $i^{th}$ copy of $q_{\infty}$ is glued to the $(i+1)^{th}$ copy of $q_0$.
\end{definition}

For $\mathbb{P}^1(s)$, we use $q_0^i$ (resp. $q_{\infty}^i$) to denote the point $q_0$ (resp. $q_{\infty}$) in the $i^{th}$ copy.
We use $Aut\big(\mathbb{P}^1(s),q_0^1,q_{\infty}^s\big)$ to denote the automorphism group of $\mathbb{P}^1(s)$ fixing $q_0^1$ and $q_{\infty}^s$.

Let $\mathbb{P}^1[s]$ denote the union of $\mathbb{P}^1(s)$ and $\mathbb{P}^1$ such that $q_{\infty}^s$ of $\mathbb{P}^1(s)$ is glued to $q_0$ of $\mathbb{P}^1$. Let $Sing(\mathbb{P}^1[s])$ denote the singular locus of $\mathbb{P}^1[s]$. Obviously, $Sing(\mathbb{P}^1[s])=\{q_{\infty}^1,\ldots,q_{\infty}^s\}$.

Let $\pi_s$ denote the contraction of $\mathbb{P}^1[s]$ to the last copy of $\mathbb{P}^1$.

Fixing $d\in \mathbb{Z}_{\geq 0}$ and a partition $\mu=\{\mu_1,\mu_2,\ldots,\mu_n\}$ of $d$. Let $\Gamma$ denote the tuple $(g,\mu,m)$, where $g,m\in \mathbb{Z}_{\geq 0}$.

\begin{definition}\label{typeI}
A stable relative map of the pair $(\mathbb{P}^1,q_0)$ with data $\Gamma$ is a tuple
\[(\Sigma,x_1,\ldots,x_m,y_1,\ldots,y_n,f,\mathbb{P}^1[s])\]
such that
\begin{itemize}
\item[(1)]$\Sigma$ is connected nodal genus $g$ curve with arithmetic genus $g$. $x_1,\ldots,y_n$ are distinct smooth marked points on $\Sigma$. We set $x=\{x_1,\ldots,x_m\}$, $y=\{y_1,\ldots,y_n\}$ and call $x$ the absolute marked points and $y$ the boundary (relative) marked points.
\item[(2)] $f$ is a morphism from $\Sigma$ to $\mathbb{P}^1[s]$, such that the degree of the composite map $\pi_s\circ f$ is $d$ and $f^{-1}(q_0^1)=\sum_{j=1}^n \mu_jy_j$.
\item[(3)] Predeformability condition: the preimage of $Sing(\mathbb{P}^1[s])$ contains only the nodes of $\Sigma$; for each node $p$ mapped to $Sing(\mathbb{P}^1[s])$, the two branches of $\Sigma$ at $p$ are mapped to two different irreducible components of $\mathbb{P}^1[s]$ and have the same contact order with $f(p)$.
\item[(4)] Stable condition: $|Aut(f)|<\infty$, where
\[Aut(f):=\{(\varphi,\psi)\in Aut(\Sigma,x,y)\times Aut(\mathbb{P}^1(s),q_0^1,q_{\infty}^s)|f\circ \varphi=\psi\circ f\}.\]
\end{itemize}
\end{definition}

Two stable relative maps $(\Sigma,x,y,f,\mathbb{P}^1[s])$ and $(\Sigma',x',y',f',\mathbb{P}^1[s'])$ are said to be isomorphic if there are isomorphisms $\varphi: (\Sigma,x,y)\rightarrow (\Sigma',x',y')$ and $\psi:\mathbb{P}^1[s]\rightarrow \mathbb{P}^1[s']$, such that $\psi\circ f=f'\circ \varphi$ and $\pi_{s'}\circ \psi=\pi_{s}$.

We use $\overline{M}_{g,m}(\mathbb{P}^1,\mu)$ to denote the moduli space of stable relative maps of the pair $(\mathbb{P}^1,q_0)$ with data $\Gamma$.

Let $L_i$ denote the cotangent line bundle associated to the absolute marked point $x_i$ and set $\psi_i=c_1(L_i)$. The evaluation map associated to $x_i$ is defined by
\begin{eqnarray*}
ev_i:  & \overline{M}_{g,m}(\mathbb{P}^1,\mu) &\longrightarrow  \mathbb{P}^1\\
       & [(\Sigma,x,y,f,\mathbb{P}^1[s])]& \longmapsto (\pi_{s}\circ f)(x_i)
\end{eqnarray*}
where $[(\Sigma,x,y,f,\mathbb{P}^1[s])]$ is an isomorphic class.

Let $\theta_1=Id,\theta_2=[pt]$ be a basis of $H^*(\mathbb{P}^1,\mathbb{Z})$, where $Id$ is the identity element and $[pt]$ is the Poincar\'{e} dual to a point. The relative Gromov-Witten invariant of the pair $(\mathbb{P}^1,q_0)$ is defined by
\begin{equation}\label{neqin}
\Big\langle\mu\Big|\prod_{i=1}^m\tau_{k_i}(\theta_{l_i})\Big\rangle_{g,d}:=\int_{[\overline{M}_{g,m}(\mathbb{P}^1,\mu)]^{vir}}\prod_{i=1}^m \psi_i^{k_i}ev_i^*(\theta_{l_i})
\end{equation}
where $[\overline{M}_{g,m}(\mathbb{P}^1,\mu)]^{vir}$ is the virtual fundamental class of $\overline{M}_{g,m}(\mathbb{P}^1,\mu)$ with virtual dimension
\begin{equation}\label{virdim}
2g-1+d+m.
\end{equation}

Next, we will introduce some particular equivariant relative Gromov-Witten invariants of the pair $(\mathbb{P}^1,q_0)$ which will be concerned in this paper.

We specialize the relative data to be
\[\Gamma=(g,\mu,m)=(0,\{d\},m)\]
where the partition $\mu=\{d\}$ which contains only one part.

We give a $\mathbb{C}^*$-action on $\mathbb{P}^1$ as follows:
\begin{eqnarray*}
 \mathbb{C}^*\times \mathbb{P}^1 & \longrightarrow & \mathbb{P}^1 \\
 (u,[z_1,z_2])  &\longmapsto & [z_1,uz_2].
\end{eqnarray*}
For $\mathbb{P}^1[s]=\mathbb{P}^1(s)\cup \mathbb{P}^1$, $\mathbb{C}^*$ will act on the last copy of $\mathbb{P}^1$ and leave $\mathbb{P}^1(s)$ invariant. This will naturally induce a $\mathbb{C}^*$-action on $\overline{M}_{0,m}(\mathbb{P}^1,\{d\})$ by composition.

Let $[\overline{M}_{0,m}(\mathbb{P}^1,\{d\})]^{vir}_{T}$ be the equivariant virtual fundamental class of $\overline{M}_{0,m}(\mathbb{P}^1,\{d\})$ with $\mathbb{C}^*$-action given as above.

Let $[\infty]$ be the equivariant class corresponding to the fixed point $q_{\infty}$. The cotangent line bundle $L_i$ has a natural equivariant lifting. By abuse of notation, we still use $\psi_i$ to denote the corresponding equivariant class.

In this paper, we will compute equivariant relative Gromov-Witten invariants in the following form:
\begin{equation}\label{keyinv}
\Big\langle d\Big |\prod_{i=1}^m \prod_{k_i=0}^{l_i}(k_i\tau+[\infty])\Big\rangle_0^{T}:=\int_{[\overline{M}_{0,m}(\mathbb{P}^1,\{d\})]^{vir}_{T}}\prod_{i=1}^m\prod_{k_i=0}^{l_i}\bigr(k_i\psi_i+ev_i^*([\infty])\bigr).
\end{equation}

We will compute (\ref{keyinv}) by relative virtual localization formula which will be given in the next subsection.

\subsection{Relative virtual localization formula}
In this subsection, we will give a description of relative virtual localization formula in our special case following \cite{FP}\cite{LLZ2}.

We give a $\mathbb{C}^*$-action to $\overline{M}_{0,m}(\mathbb{P}^1,\{d\})$ as in the subsection \ref{pre1}.

For a $\mathbb{C}^*$-fixed point
\begin{equation}\label{fixpt}
\Big[\Big(\Sigma,x,y,f,\mathbb{P}^1[s]\Big)\Big]
\end{equation}
of $\overline{M}_{0,m}(\mathbb{P}^1,\{d\})$, we may decompose $\Sigma$ as
\[\Sigma_0\cup \Sigma_{\infty}\cup\Big(\bigcup_{r=1}^k\Sigma_{e_r}\Big)\]
where $\Sigma_0=(\pi_s\circ f)^{-1}(q_0)$, $\Sigma_{\infty}=(\pi_s\circ f)^{-1}(q_{\infty})$ and each irreducible component $\Sigma_{e_r}$ is a rational sphere. The restriction of $\pi_s\circ f$ to $\Sigma_{e_r}$ is a cover of $\mathbb{P}^1$ with full ramification over $q_0$ and $q_{\infty}$.

Since $f$ is totally ramified over $q_0^1$, we may deduce that $\Sigma_0$ contains only one connected component. So each $\Sigma_{e_r}$ intersects with $\Sigma_0$ at one point. Since the total genus $g(\Sigma)=0$, we may further deduce that $g(\Sigma_0)=0$. The constraint that $g(\Sigma)=0$ also implies that $\Sigma_{\infty}$ has $k$ connected components. We may label them as $\Sigma_1,\Sigma_2,\ldots,\Sigma_k$. Then each $\Sigma_i$ intersects with one of $\Sigma_{e_r}$ at one point and $g(\Sigma_i)=0$. Without loss of generality, we may assume that $\Sigma_i$ intersects with $\Sigma_{e_i}$ at one point.

Now we can associate a localization graph $G_l$ to the fixed point (\ref{fixpt}). It consists of
\begin{itemize}
  \item[(I)] A set $V$ of vertices $v_0,v_1,\ldots,v_k$ where $v_i$ corresponds to connected component $\Sigma_i$.
  \item[(II)] A set $E$ of edges $e_1,\ldots,e_k$ where $e_r$ corresponds to $\Sigma_{e_r}$ and $v_0$ connects to $v_r$ via $e_r$.
  \item[(III)] An assignment of degrees of $\nu:E\rightarrow \mathbb{Z}_{>0}$ given by degrees of the covers. Obviously, we have $\sum_{r=1}^k \nu(e_r)=d$.
  \item[(IV)] An assignment of absolute marked points $a:\{1,2,\ldots,m\}\rightarrow V$. We also assign the only relative marked point $y_1$ to $v_0$.
\end{itemize}

The fixed points of $\overline{M}_{0,m}(\mathbb{P}^1,\{d\})$ with the same localization graph $G_l$ form a connected component of the fixed loci. We denote it as $\overline{F}_{G_l}$.

The valence $val(v_i)$ of a vertex $v_i$ is defined to be the number of all marked points and edges associated to $v_i$. For $i>0$, it is easy to see that $val(v_i)=|a^{-1}(v_i)|+1$.

Next, we will give a description of $\overline{F}_{G_l}$ according to the target $\mathbb{P}^1[s]$.

Case (i), suppose that the target for a general morphism in $\overline{F}_{G_l}$ is $\mathbb{P}^1[s]$ with $s>0$.

Let $\overline{M}_{G_l}^{\sim}$ denote the rubber space which consists of morphisms
\[\tilde{f}:\widetilde{\Sigma}\longrightarrow \mathbb{P}^1(s)\]
such that
\begin{itemize}
  \item[(A)] $\widetilde{\Sigma}$ is a connected, genus $0$, nodal curve with absolute marked points $x_i$ indexed by those $i$ such that $a(i)=v_0$ and relative marked points $y_1$, $z_1,z_2,\ldots,z_k$.
  \item[(B)] As Cartier divisors
  \[\tilde{f}^{-1}(q_0^1)=dy_1,~~~\tilde{f}^{-1}(q_{\infty}^{s})=\sum_{r=1}^k \nu(e_r)z_r.\]
  \item[(C)]Predeformability and stable conditions similar to the definition of \ref{typeI}.
\end{itemize}
Let $C_l$ be an irreducible component of $\mathbb{P}^1(s)$. It is easy to deduce from condition (B) and predeformability condition that
\[\tilde{f}\bigr|_{\tilde{f}^{-1}(C_l)}:\tilde{f}^{-1}(C_l)\longrightarrow C_l\simeq \mathbb{P}^1\]
is a degree $d$ map. $\overline{M}_{G_l}^{\sim}$ has a virtual fundamental class $[\overline{M}_{G_l}^{\sim}]^{vir}$ whose virtual dimension is
\begin{equation}\label{virdim-rb}
|a^{-1}(v_0)|+k-2.
\end{equation}

Now the restriction of $f$ to $\Sigma_0$ can be seen as an element of $\overline{M}_{G_l}^{\sim}$. Since $f$ maps $\Sigma_{\infty}$ to the point $q_{\infty}$, the restriction of $f$ to $\Sigma_{\infty}$ can be treated as an element of
\[\overline{M}_{G_l}^{\infty}=\prod_{r=1}^k \overline{M}_{0,val(v_r)}.\]
Here, the possible unstable moduli space $\overline{M}_{0,1}$ or $\overline{M}_{0,2}$ in $\overline{M}_{G_l}^{\infty}$ should be treated as a point. We set
\[\overline{M}_{G_l}=\overline{M}_{G_l}^{\infty}\times \overline{M}_{G_l}^{\sim}.\]
The moduli space $\overline{M}_{G_l}$ has a virtual fundamental class which is given by
\[[\overline{M}_{G_l}]^{vir}=[\overline{M}_{G_l}^{\infty}]\times [\overline{M}_{G_l}^{\sim}]^{vir}=\prod_{r=1}^k [\overline{M}_{0,val(v_r)}]\times [\overline{M}_{G_l}^{\sim}]^{vir}.\]

There is a natural group $\mathbf{A}_{G_l}$ (see \cite{FP} Section 1.3.4 for more details) acting on $\overline{M}_{G_l}$, such that it is filtered by an exact sequence of groups:
\[1\rightarrow\prod_{r=1}^k \mathbb{Z}_{\nu(e_r)}\rightarrow\mathbf{A}_{G_l}\rightarrow Aut(G_l)\rightarrow 1\]
where the automorphism group $Aut(G_l)$ consists of those automorphisms of $G_l$ which leave all the assignments invariant. Then we have a natural identification
\[\overline{M}_{G_l}/\mathbf{A}_{G_l}\simeq\overline{F}_{G_l}.\]
We further set
\[\tau_{G_l}:\overline{M}_{G_l}\longrightarrow \overline{F}_{G_l}\]
to be natural quotient map.

Case (ii), suppose that the target is $\mathbb{P}^1$. Then $\Sigma_0$ degenerates to be a point. So there is only one edge and one component in $\Sigma_{\infty}$. It further implies that the degree of the only edge $e_1$ is $d$ and all the absolute marked points are assigned to the vertex $v_1$. So $val(v_1)=m+1$.

Now if we set
\[\overline{M}_{G_l}=\overline{M}_{0,m+1},~~~\mathbf{A}_{G_l}=\mathbb{Z}_d,\]
then we still have the quotient map
\[\tau_{G_l}:\overline{M}_{G_l}\longrightarrow \overline{F}_{G_l}\simeq \overline{M}_{G_l}/\mathbf{A}_{G_l}\]
where $\mathbf{A}_{G_l}$ acts trivially on $\overline{M}_{G_l}$. In this case
\[[\overline{M}_{G_l}]^{vir}=[\overline{M}_{0,m+1}].\]

Let $N_{G_l}^{vir}$ denote the virtual normal bundle associated to $\overline{F}_{G_l}$. The equivariant Euler class of virtual normal bundle $e_T(N_{G_l}^{vir})$ plays an important role in the relative virtual localization formula. We will give a description of the inverse of $e_T(N_{G_l}^{vir})$ following \cite{FP}\cite{Ga}\cite{GV}\cite{LLZ2}.

Case (i), suppose that the target is $\mathbb{P}^1[s]$ with $s>0$.

The contribution of each edge $e_r$ to the inverse of $e_T(N_{G_l}^{vir})$ is given by
\[\frac{t}{t^{\nu(e_r)}\frac{\nu(e_r)!}{\nu(e_r)^{\nu(e_r)}}}.\]

For each vertex $v_r$ with $r>0$, we will discuss the contribution of vertex $v_r$ to the inverse of $e_T(N_{G_l}^{vir})$ according to $val(v_r)$.
\begin{itemize}
  \item[(A)] If $val(v_r)=1$, then the contribution is
  \[\frac{1}{N(v_r)}:=\frac{1}{\nu(e_r)}.\]
  \item[(B)] If $val(v_r)=2$, then the contribution is
  \[\frac{1}{N(v_r)}:=\frac{1}{t}.\]
  \item[(C)] If $val(v_r)\geq 3$, then the contribution is
  \[\frac{1}{N(v_r)}:=\frac{1}{t(\frac{t}{\nu(e_r)}-\psi_{e_r})},\]
 where $\psi_{e_r}$ is the $\psi$-class of $\overline{M}_{0,val(v_r)}$ associated to the marked point coming from edge $e_r$.
\end{itemize}

There is one contribution coming from deforming the target singularity, which is given by
\[\frac{\prod_{r=1}^k \nu(e_r)}{-t-\psi_{\infty}}.\]
We will explain the notation $\psi_{\infty}$ in the next. The cotangent line at the point $q_{\infty}^s$ of $\mathbb{P}^1(s)$ will induce a line bundle $L_{\infty}$ on $\overline{M}_{G_l}^{\sim}$. We then use $\psi_{\infty}$ to denote the first Chern class of $L_{\infty}$ (see \cite{GV} Section 2.5 for more details of $\psi_{\infty}$).

Since $\mathbb{C}^*$ acts only on the second factor of $\mathbb{P}^1[s]=\mathbb{P}^1(s)\cup \mathbb{P}^1$, there will be no contributions from vertex $v_0$.

From the above discussion, we may conclude that
\[\frac{1}{e_T(N_{G_l}^{vir})}=\frac{\prod_{r=1}^k \nu(e_r)}{-t-\psi_{\infty}}\prod_{r=1}^k\frac{t}{t^{\nu(e_r)}\frac{\nu(e_r)!}{\nu(e_r)^{\nu(e_r)}}}\frac{1}{N(v_r)}.\]

Case (ii), suppose that the target is $\mathbb{P}^1$. Since there is only one edge $e_1$ with degree $d$, the contribution coming from edge becomes
\[\frac{t}{t^d\frac{d!}{d^d}}.\]
As for the vertex $v_1$, the contribution is given by
\[\frac{1}{N(v_1)}\]
where $1/N(v_1)$ can be determined as in case (i). There will be no contribution coming from deforming the target singularity. So we have
\[\frac{1}{e_T(N_{G_l}^{vir})}=\frac{t}{t^d\frac{d!}{d^d}}\frac{1}{N(v_1)}.\]

The relative virtual localization formula expresses the equivariant virtual fundamental class of $\overline{M}_{0,m}(\mathbb{P}^1,\{d\})$ in terms of contributions from each localization graph $G_l$:
\begin{equation}\label{rvlf}
[\overline{M}_{0,m}(\mathbb{P}^1,\{d\})]^{vir}_{T}=\sum_{G_l}\frac{1}{|\mathbf{A}_{G_l}|}(\tau_{G_l})_*\Bigg(\frac{[\overline{M}_{G_l}]^{vir}}{e_T(N_{G_l}^{vir})}\Bigg).
\end{equation}

\section{Proofs}\label{2}
We will prove Theorems \ref{keythm1}, \ref{keythm2} and \ref{keythm3} in this section. We will always assume that $m\geq 2$ in the following.

Let
\[\omega_T=\prod_{i=1}^m\prod_{k_i=0}^{l_i}\bigr(k_i\psi_i+ev_i^*([\infty])\bigr).\]
Then by relative virtual localization formula (\ref{rvlf}), we know that (\ref{keyinv}) is given by
\begin{equation}\label{contri}
\sum_{G_l}\frac{1}{|\mathbf{A}_{G_l}|}(p_T\circ\tau_{G_l})_*\Bigg(\frac{\tau_{G_l}^*(\omega_T)\cap[\overline{M}_{G_l}]^{vir}}{e_T(N_{G_l}^{vir})}\Bigg)
\end{equation}
where $p_T$ is the equivariant push-forward to a point.

If the marked point $x_i$ is assigned to the vertex $v_0$ of $G_l$, then the composite map
\[ev_i\circ\tau_{G_l}:\overline{M}_{G_l}\longrightarrow q_0\longrightarrow \mathbb{P}^1\]
factors through the fixed point $q_0$. Since the restriction of the equivariant class $[\infty]$ to the fixed point $q_0$ becomes zero, we have $(ev_i\circ\tau_{G_l})^*([\infty])=0$. It further implies that $\tau_{G_l}^*(\omega_T)=0$.

So in the following, we only consider those localization graphs $G_l$ such that all the absolute marked points are assigned to vertices $v_r$ with $r>0$. Then $|a^{-1}(v_0)|=0$ and the composite map
\[ev_i\circ\tau_{G_l}:\overline{M}_{G_l}\longrightarrow q_{\infty}\longrightarrow \mathbb{P}^1\]
factors through the fixed point $q_{\infty}$.

Firstly, let us consider those $G_l$ such that the target for $\overline{F}_{G_l}$ is of the form $\mathbb{P}^1[s]$ with $s>0$. We set
\[\lambda_r:=|a^{-1}(v_r)|\]
to be the number of absolute marked points assigned to $v_r$, where $1\leq r\leq k$. Without loss of generality, we may assume that
\[\lambda_1\geq \lambda_2\geq\ldots\lambda_{k_1}\geq 2>\lambda_{k_1+1}= \ldots\lambda_{k_2}=1>\lambda_{k_2+1}=\ldots\lambda_k=0\]
where $0\leq k_1\leq k_2\leq k$. We may further assume that
\[a^{-1}(v_r)=\Big\{i_1^{r},i_2^{r},\ldots,i_{\lambda_r}^{r}\Big\}.\]

Recall that $\nu:E\rightarrow \mathbb{Z}_{>0}$ is the assignment of degrees. Here, for simplicity we set
\[d_r:=\nu(e_r),~~~1\leq r\leq k.\]
Then we have
\[|\mathbf{A}_{G_l}|=|Aut(G_l)|\prod_{r=1}^k d_r.\]
In this case, the automorphism group $Aut(G_l)$ is just the group of permutation symmetries of the set
\[\{d_{k_2+1},\ldots,d_{k}\}.\]

As for $\overline{M}_{G_l}$, we have
\[\overline{M}_{G_l}=\overline{M}_{G_l}^{\sim}\times \prod_{r=1}^{k_1}\overline{M}_{0,val(v_r)}=\overline{M}_{G_l}^{\sim}\times \prod_{r=1}^{k_1}\overline{M}_{0,\lambda_r+1}.\]
Here, we omit those $\overline{M}_{0,val(v_r)}$ such that $r>k_1$ since they degenerate to be points.

In this case, we may write $\tau_{G_l}^*(\omega_T)$ as
\begin{equation*}
\prod_{r=1}^k\prod_{j=1}^{\lambda_r}\prod_{k_{i_j^{r}}=0}^{l_{i_j^{r}}}\Big(k_{i_j^{r}}\tau_{G_l}^*\bigr(\psi_{i_j^{r}}\bigr)+\tau_{G_l}^*ev_{i_j^{r}}^*([\infty])\Big).
\end{equation*}

Since $ev_{i_j^{r}}\circ\tau_{G_l}$ factors through $q_{\infty}$, we know that
\[\tau_{G_l}^*ev_{i_j^{r}}^*([\infty])=t.\]

As for $\tau_{G_l}^*\bigr(\psi_{i_j^{r}}\bigr)$, it depends on $\lambda_r$. We need $\lambda_r>0$. Otherwise it does not appear.

Case (a), suppose that $\lambda_r\geq 2$. In this case, $\tau_{G_l}^*\bigr(\psi_{i_j^{r}}\bigr)$ becomes a $\psi$-class of $\overline{M}_{0,\lambda_r+1}$. We may simply set
\begin{equation}\label{psino}
\psi_j^{r}:=\tau_{G_l}^*\bigr(\psi_{i_j^{r}}\bigr).
\end{equation}

Case (b), suppose that $\lambda_r=1$. In this case, the connected component corresponding to $v_r$ degenerates to be the marked point $x_{i_1^r}$. It sits on the rational component $\Sigma_{e_r}$.

We recall that $\Sigma_{e_r}$ is mapped to $\mathbb{P}^1$ with degree $d_r$ and full ramification over $q_0$ and $q_{\infty}$. So we have a natural identification

\[\tau_{G_l}^*\bigr(L_{i_1^r}^{\otimes d_r}\bigr)\simeq\mathcal{O}\]
where $\mathcal{O}$ is the trivial bundle with $\mathbb{C}^*$-action given by scaling. It further implies that
\[\tau_{G_l}^*\bigr(\psi_{i_1^{r}}\bigr)=-\frac{t}{d_r}.\]
Here, we recall that the equivalent class $t$ corresponds to the dual of the standard representation of $\mathbb{C}^*$.

Now we may conclude that
\[\tau_{G_l}^*(\omega_T)=\Bigg\{\prod_{r=1}^{k_1}\prod_{j=1}^{\lambda_r}\prod_{k_{i_j^{r}}=0}^{l_{i_j^{r}}}\bigr(k_{i_j^{r}}\psi_j^{r}+t\bigr)\Bigg\}\Bigg\{\prod_{r=k_1+1}^{k_2}\prod_{k_{i_1^{r}}=0}^{l_{i_1^{r}}}\Big(1-\frac{k_{i_1^{r}}}{d_r}\Big)t\Bigg\}.\]
It is easy to see that if $l_{i_1^r}\geq d_r$ for some $k_1+1\leq r\leq k_2$, then we have $\tau_{G_l}^*(\omega_T)=0$. So we may assume that $l_{i_1^r}+1\leq d_r$ for all $k_1+1\leq r\leq k_2$.

As for $1/e_T(N_{G_l}^{vir})$, it becomes
\[\frac{\prod_{r=1}^kd_r}{-t-\psi_{\infty}}\Bigg\{\prod_{r=1}^k\frac{t}{t^{d_r}\frac{d_r!}{d_r^{d_r}}}\Bigg\}\Bigg\{\prod_{r=1}^{k_1}\frac{1}{t(\frac{t}{d_r}-\psi_{e_r})}\Bigg\}\frac{1}{t^{k_2-k_1}}\frac{1}{\prod_{r=k_2+1}^k d_r}.\]
We may treat $\psi_{e_r}$ as the $\psi$-class of the last marked point in $\overline{M}_{0,\lambda_r+1}$. So by using the notation as for (\ref{psino}), we may simply set $\psi_{\lambda_r+1}^{r}:=\psi_{e_r}$.

From the above discussion, we may conclude that term
\begin{equation}\label{evaterm}
\frac{1}{|\mathbf{A}_{G_l}|}(p_T\circ\tau_{G_l})_*\Bigg(\frac{\tau_{G_l}^*(\omega_T)\cap[\overline{M}_{G_l}]^{vir}}{e_T(N_{G_l}^{vir})}\Bigg)
\end{equation}
in (\ref{contri}) equals to
\begin{equation}\label{keyterm}
\begin{aligned}
&t^{-d+k+\sum_{r=k_1+1}^{k_2}l_{i_1^r}}\prod_{r=1}^{k_1}\frac{d_r^{d_r}}{d_r!}\prod_{r=k_1+1}^{k_2}\frac{d_r^{d_r-l_{i_1^r}-1}}{(d_r-l_{i_1^r}-1)!}\times\\
&~~~~~~\frac{1}{|Aut(G_l)|}\prod_{r=k_2+1}^k\frac{d_r^{d_r-1}}{d_r!}\times\\
&~~~~~~~~~\int_{[\overline{M}_{G_l}^{\sim}]^{vir}}\frac{1}{-t-\psi_{\infty}}\prod_{r=1}^{k_1}\int_{[\overline{M}_{0,\lambda_r+1}]}\frac{\prod_{j=1}^{\lambda_r}\prod_{k_{i_j^r}=0}^{l_{i_j^r}}\bigr(k_{i_j^r}\psi_j^r+t\bigr)}{t(\frac{t}{d_r}-\psi_{\lambda_r+1}^r)}.
\end{aligned}
\end{equation}
Here, in order to calculate the integrals
\begin{equation}\label{midterm}
\int_{[\overline{M}_{G_l}^{\sim}]^{vir}}\frac{1}{-t-\psi_{\infty}},~~~\int_{[\overline{M}_{0,\lambda_r+1}]}\frac{\prod_{j=1}^{\lambda_r}\prod_{k_{i_j^r}=0}^{l_{i_j^r}}\bigr(k_{i_j^r}\psi_j^r+t\bigr)}{t(\frac{t}{d_r}-\psi_{\lambda_r+1}^r)}.
\end{equation}
We should firstly expand them according to $t$ and then integrant the coefficients.

The integrals (\ref{midterm}) can be computed by the following two lemmas.

\begin{lemma}\label{cal1}
\[\int_{[\overline{M}_{G_l}^{\sim}]^{vir}}\frac{1}{-t-\psi_{\infty}}=(-t)^{1-k}d^{k-2}.\]
\end{lemma}

\begin{proof}
Since $|a^{-1}(v_0)|=0$, by (\ref{virdim-rb}) we have
\[dim([\overline{M}_{G_l}^{\sim}]^{vir})=k-2.\]
So
\[\int_{[\overline{M}_{G_l}^{\sim}]^{vir}}\frac{1}{-t-\psi_{\infty}}=(-t)^{1-k}\int_{[\overline{M}_{G_l}^{\sim}]^{vir}}\psi_{\infty}^{k-2}.\]
Now by \cite{Ga}, Lemma 5.3.1, we have
\begin{equation}\label{mid-rb}
\int_{[\overline{M}_{G_l}^{\sim}]^{vir}}\psi_{\infty}^{k-2}=d^{k-2}.
\end{equation}
The lemma directly follows.
\end{proof}

We remark that we can also compute the L.H.S of (\ref{mid-rb}) by \cite{LLZ1}, Proposition 5.5 which relates it to one-part double Hurwitz numbers. The latter can further be computed by \cite{GJV}, Theorem 3.1.

For $l\in \mathbb{Z}_{\geq 0}$, we set
\[(w+1)_l=
\begin{cases}
(w+1)(w+2)\ldots (w+l), ~~~\text{if~} l>0,\\
1,~~~\text{if ~} l=0.
\end{cases}
\]
We define $A_l^b$ by
\begin{equation}\label{stir}
\sum_{b=0}^l A_l^b w^b:=(w+1)_l.
\end{equation}
Actually, $A_l^b$ are Stirling numbers (See \cite{C}, Chapter 8 for more details).

\begin{lemma}\label{cal2}
\[
\begin{aligned}
\int_{[\overline{M}_{0,n+1}]}&\frac{\prod_{j=1}^n\prod_{k_j=0}^{l_j}(k_j\psi_j+t)}{t(\frac{t}{D}-\psi_{n+1})}=\\
&Dt^{\sum_{j=1}^n l_j}[w_1^{l_1}w_2^{l_2}\ldots w_n^{l_n}]\prod_{j=1}^n (w_j+1)_{l_j}\Big(\sum_{j=1}^n w_j+D\Big)^{n-2}.
\end{aligned}
\]
\end{lemma}
\begin{proof}
By the definition of $A_l^b$, we have
\[\prod_{k_j=0}^{l_j}(k_j\psi_j+t)=\sum_{s_j=0}^{l_j}t^{s_j+1}A_{l_j}^{s_j}\psi_j^{l_j-s_j}.\]
So
\[\prod_{j=1}^n\prod_{k_j=0}^{l_j}(k_j\psi_j+t)=\sum_{s_1,\ldots s_n \atop 0\leq s_j\leq l_j}t^{\sum_{j=1}^n s_j+n}\prod_{j=1}^n A_{l_j}^{s_j}\prod_{j=1}^n \psi_j^{l_j-s_j}.\]
We also have
\[\frac{1}{t(\frac{t}{D}-\psi_{n+1})}=\sum_{s\geq 0}t^{-s-2}D^{s+1}\psi_{n+1}^s.\]
So we have
\[
\begin{aligned}
\int_{[\overline{M}_{0,n+1}]}&\frac{\prod_{j=1}^n\prod_{k_j=0}^{l_j}(k_j\psi_j+t)}{t(\frac{t}{D}-\psi_{n+1})}=\\
&\sum_{s_1,\ldots s_n \atop 0\leq s_j\leq l_j}\sum_{s\geq 0}t^{\sum_{j=1}^n s_j+n-s-2}D^{s+1}\prod_{j=1}^n A_{l_j}^{s_j}\int_{[\overline{M}_{0,n+1}]}\psi_{n+1}^s\prod_{j=1}^n\psi_j^{l_j-s_j}.\\
\end{aligned}
\]
Since $dim([\overline{M}_{0,n+1}])=n-2$, we need
\[s+\sum_{j=1}^n (l_j-s_j)=n-2.\]
In that case, we have
\[\int_{[\overline{M}_{0,n+1}]}\psi_{n+1}^s\prod_{j=1}^n\psi_j^{l_j-s_j}=\frac{(n-2)!}{s!\prod_{j=1}^n(l_j-s_j)!}\]
by string and dilaton equations.

From the above discussion, we may conclude that
\[
\begin{aligned}
\int_{[\overline{M}_{0,n+1}]}&\frac{\prod_{j=1}^n\prod_{k_j=0}^{l_j}(k_j\psi_j+t)}{t(\frac{t}{D}-\psi_{n+1})}=Dt^{\sum_{j=1}^n l_j}\times\\
&\sum_{0\leq s_j\leq l_j \atop \sum s_j\leq n-2}\prod_{j=1}^n A_{l_j}^{l_j-s_j}D^{n-2-\sum_{j=1}^ns_j}\frac{(n-2)!}{(n-2-\sum_{j=1}^ns_j)!\prod_{j=1}^ns_j!}.
\end{aligned}
\]
It is easy to see that
\[\sum_{0\leq s_j\leq l_j \atop \sum s_j\leq n-2}\prod_{j=1}^n A_{l_j}^{l_j-s_j}D^{n-2-\sum_{j=1}^ns_j}\frac{(n-2)!}{(n-2-\sum_{j=1}^ns_j)!\prod_{j=1}^ns_j!}\]
gives the coefficient of the term $w_1^{l_1}w_2^{l_2}\ldots w_n^{l_n}$ in
\[\prod_{j=1}^n (w_j+1)_{l_j}\Big(\sum_{j=1}^n w_j+D\Big)^{n-2}.\]
The lemma directly follows.
\end{proof}

Before we continue the computation, we will introduce some notation. We set
\[\mathcal{I}^{r}_{G_l}=a^{-1}(v_r),~~~1\leq r\leq k_2.\]
Obviously, $\mathcal{I}^1_{G_l}, \mathcal{I}^2_{G_l}\ldots,\mathcal{I}^{k_2}_{G_l}$ form a partition of $\{1,2,\ldots,m\}$. We further set
\[I_{G_l}:=\bigcup_{r=k_1+1}^{k_2}\mathcal{I}^r_{G_l}\]
i.e. an union of those $\mathcal{I}^r_{G_l}$ which contain only one element. Let
\[T_{G_l}^{\geq 2}=\{\mathcal{I}^1_{G_l},\mathcal{I}^2_{G_l},\ldots,\mathcal{I}^{k_1}_{G_l}\}\]
be a collection of those $\mathcal{I}^r_{G_l}$ which contain at least two elements. So $T_{G_l}^{\geq 2}$ is a partition of the complement $I^{c}_{G_l}$ of $I_{G_l}$ in $\{1,2,\ldots,m\}$.

Now applying Lemmas \ref{cal1} and \ref{cal2}, term (\ref{keyterm}) becomes
\begin{equation}\label{finterm}
\begin{aligned}
&t^{\sum_{i=1}^m l_i-d+1}\Big[\prod_{i\in I_{G_l}^c}w_i^{l_i}\Big]\prod_{i\in I_{G_l}^c}(w_i+1)_{l_i}(-1)^{k_1+|I_{G_l}|-1}d^{k_1+|I_{G_l}|-2}\times\\
&\prod_{r=k_1+1}^{k_2}\frac{d_r^{d_r-l_{i_1^r}-1}}{(d_r-l_{i_1^r}-1)!}\times\\
&\Bigg\{\frac{1}{|Aut(G_l)|}\prod_{r=k_2+1}^k\frac{-dd_r^{d_r-1}}{d_r!}\Bigg\}\times \prod_{r=1}^{k_1}\frac{d_r^{d_r+1}}{d_r!}\Bigg(\sum_{i\in \mathcal{I}^r_{G_l}}w_{i} +d_r\Bigg)^{|\mathcal{I}^r_{G_l}|-2}.\\
\end{aligned}
\end{equation}
We note that here $d_r- l_{i_1^r}-1\geq 0$ and $\sum_{r=1}^k d_r=d$.

We have computed term (\ref{evaterm}) when the target for $\overline{F}_{G_l}$ is of the form $\mathbb{P}^1[s]$ with $s>0$. When the target is $\mathbb{P}^1$, by a similar computation as above, (\ref{evaterm}) becomes
\[
t^{\sum_{i=1}^m l_i-d+1}\Big[\prod_{i=1}^m w_i^{l_i}\Big]\prod_{i=1}^m(w_i+1)_{l_i}\frac{d^d}{d!}\Bigg(\sum_{i=1}^m w_i+d\Bigg)^{m-2}\\
\]
which coincides with (\ref{finterm}) if we set $I_{G_l}=\emptyset$ and $k_1=k_2=k=1$.

Choosing a subset $I\subset \{1,2,\ldots,m\}$. We want to compute the summation
\begin{equation}\label{crstm}
\sum_{\{G_l:I_{G_l}=I\}}\frac{1}{|\mathbf{A}_{G_l}|}(p_T\circ\tau_{G_l})_*\Bigg(\frac{\tau_{G_l}^*(\omega_T)\cap[\overline{M}_{G_l}]^{vir}}{e_T(N_{G_l}^{vir})}\Bigg)
\end{equation}

Firstly, we suppose that $|I|\leq m-1$. Then (\ref{crstm}) equals to
\[\sum_{\{\mathcal{I}_1,\ldots,\mathcal{I}_s\}\in P_{I^c}^{\geq 2}}\sum_{\bigr\{G_l:T_{G_l}^{\geq 2}=\{\mathcal{I}_1,\ldots,\mathcal{I}_s\}\bigr\}}\frac{1}{|\mathbf{A}_{G_l}|}(p_T\circ\tau_{G_l})_*\Bigg(\frac{\tau_{G_l}^*(\omega_T)\cap[\overline{M}_{G_l}]^{vir}}{e_T(N_{G_l}^{vir})}\Bigg)\]
by definition. Here, $P_{I^c}^{\geq 2}$ is given by definition \ref{par2}. By (\ref{finterm}), it is easy to deduce that for a fixed partition $\{\mathcal{I}_1,\ldots,\mathcal{I}_s\}\in P_{I^c}^{\geq 2}$, the summation
\begin{equation}\label{prefm}
\sum_{\bigr\{G_l:T_{G_l}^{\geq 2}=\{\mathcal{I}_1,\ldots,\mathcal{I}_s\}\bigr\}}\frac{1}{|\mathbf{A}_{G_l}|}(p_T\circ\tau_{G_l})_*\Bigg(\frac{\tau_{G_l}^*(\omega_T)\cap[\overline{M}_{G_l}]^{vir}}{e_T(N_{G_l}^{vir})}\Bigg)
\end{equation}
equals to
\begin{equation*}
\begin{aligned}
&t^{\sum_{i=1}^ml_i-d+1}\Bigg[\prod_{j\in I^c}w_j^{l_j}q^{d-\sum_{i\in I}(l_i+1)}\Bigg]\prod_{j\in I^c}(w_j+1)_{l_j}(-1)^{s+|I|-1}d^{s+|I|-2}\times\\
&\prod_{i\in I}\Bigg(\sum_{n=0}^{\infty}\frac{(n+l_i+1)^n}{n!}q^n\Bigg)\times\\
&e^{-d\sum_{n=1}^{\infty}\frac{n^{n-1}}{n!}q^n}\prod_{r=1}^s\Bigg(\sum_{n=1}^{\infty}\frac{n^{n+1}}{n!}\Big(\sum_{j\in \mathcal{I}_r}w_j+n\Big)^{|\mathcal{I}_r|-2}q^n\Bigg).
\end{aligned}
\end{equation*}

By \cite{CGHJK} Formula (2.37), we know that
\begin{equation}\label{keyid}
\frac{e^{\mathcal{W}(q)x}}{1-\mathcal{W}(q)}=\sum_{n=0}^{\infty}(n+x)^n\frac{q^n}{n!},~~~x\in\mathbb{R},
\end{equation}
where $\mathcal{W}(q)$ is given by (\ref{lamW}). So
\[\prod_{i\in I}\Bigg(\sum_{n=0}^{\infty}\frac{(n+l_i+1)^n}{n!}q^n\Bigg)=\frac{e^{\sum_{i\in I}(l_i+1)\mathcal{W}(q)}}{(1-\mathcal{W}(q))^{|I|}}\]
and
\[e^{-d\sum_{n=1}^{\infty}\frac{n^{n-1}}{n!}q^n}=e^{-d\mathcal{W}(q)}.\]
As for
\begin{equation}\label{term3}
\prod_{r=1}^s\Bigg(\sum_{n=1}^{\infty}\frac{n^{n+1}}{n!}\Big(\sum_{j\in \mathcal{I}_r}w_j+n\Big)^{|\mathcal{I}_r|-2}q^n\Bigg),
\end{equation}
it is easy to deduce from (\ref{keyid}) that it becomes
\[\prod_{r=1}^s\Biggr\{\Bigg(q\frac{d}{dq}+\sum_{j\in \mathcal{I}_r}w_j\Bigg)^{|\mathcal{I}_r|-2}\circ\Bigg(q\frac{d}{dq}\Bigg)\Bigg(\frac{1}{1-\mathcal{W}(q)}\Bigg)\Biggr\}\]
where
\[\Bigg(q\frac{d}{dq}+\sum_{j\in \mathcal{I}_r}w_j\Bigg)^{|\mathcal{I}_r|-2}\circ\Bigg(q\frac{d}{dq}\Bigg)\]
should be seen as an operator which acts on $\frac{1}{1-\mathcal{W}(q)}$. Recall that the famous Lambert W function is
\[W(q)=\sum_{n=1}^{\infty}\frac{(-n)^{n-1}}{n!}q^n.\]
It satisfies $q=W(q)e^{W(q)}$. So we may deduce that
\begin{equation}\label{invsrls}
W(q)=-\mathcal{W}(-q),~~~q=\mathcal{W}(q)e^{-\mathcal{W}(q)}.
\end{equation}
Now it is easy to deduce that
\[q\frac{d}{dq}=\frac{\mathcal{W}}{1-\mathcal{W}}\frac{d}{d\mathcal{W}}.\]
So (\ref{term3}) becomes
\[\prod_{r=1}^s\mathcal{O}_{
\sum_{j\in \mathcal{I}_r}w_j}^{|\mathcal{I}_r|-2}\Bigg(\frac{\mathcal{W}}{(1-\mathcal{W})^3}\Bigg)\]
where operators $\mathcal{O}_{\sum_{j\in \mathcal{I}_r}w_j}$ can be defined by (\ref{oper}).

From the above discussion, (\ref{prefm}) becomes
\begin{equation*}
\begin{aligned}
&t^{\sum_{i=1}^ml_i-d+1}\Bigg[\prod_{j\in I^c}w_j^{l_j}q^{d-\sum_{i\in I}(l_i+1)}\Bigg]\prod_{j\in I^c}(w_j+1)_{l_j}(-1)^{s+|I|-1}d^{s+|I|-2}\times\\
&\frac{e^{-\big(d-\sum_{i\in I}(l_i+1)\big)\mathcal{W}(q)}}{(1-\mathcal{W}(q))^{|I|}} \prod_{r=1}^s\mathcal{O}_{
\sum_{j\in \mathcal{I}_r}w_j}^{|\mathcal{I}_r|-2}\Bigg(\frac{\mathcal{W}}{(1-\mathcal{W})^3}\Bigg).
\end{aligned}
\end{equation*}
So (\ref{crstm}) becomes
\begin{equation}\label{imptm}
\begin{aligned}
&t^{\sum_{i=1}^ml_i-d+1}\sum_{\{\mathcal{I}_1,\ldots,\mathcal{I}_s\}\in P_{I^c}^{\geq 2}}\Bigg[\prod_{j\in I^c}w_j^{l_j}q^{d-\sum_{i\in I}(l_i+1)}\Bigg]\prod_{j\in I^c}(w_j+1)_{l_j}\times\\
&(-1)^{s+|I|-1}d^{s+|I|-2}\frac{e^{-\big(d-\sum_{i\in I}(l_i+1)\big)\mathcal{W}(q)}}{(1-\mathcal{W}(q))^{|I|}} \prod_{r=1}^s\mathcal{O}_{
\sum_{j\in \mathcal{I}_r}w_j}^{|\mathcal{I}_r|-2}\Bigg(\frac{\mathcal{W}}{(1-\mathcal{W})^3}\Bigg).
\end{aligned}
\end{equation}

Next, we suppose that $|I|=m$. So $I=\{1,2,\ldots,m\}$. Now it is easy to compute from (\ref{finterm}) that (\ref{crstm}) becomes
\begin{equation}\label{s-imptm}
t^{\sum_{i=1}^ml_i-d+1}[q^{d-\sum_{i=1}^m(l_i+1)}](-1)^{m-1}d^{m-2}\frac{e^{-\big(d-\sum_{i=1}^m(l_i+1)\big)\mathcal{W}(q)}}{(1-\mathcal{W}(q))^{m}}.
\end{equation}

By comparing (\ref{imptm}), (\ref{s-imptm}) with (\ref{R_I}), (\ref{s-R_I}), we know that $R_I$ can be obtained from (\ref{crstm}) by setting $t=1,l_1=x_1,\ldots,l_m=x_m,d=y$.

Now by (\ref{contri}) and (\ref{crstm}), it is easy to deduce that
\begin{equation}\label{keyfm}
F(x_1,\ldots,x_m,y)=\sum_{I\subset\{1,2,\ldots,m\}}R_I,~~~(x_1,\ldots,x_m,y)\in S_m\cap\mathbb{Z}^{m+1}.
\end{equation}
Next, we claim that if $|I|>m-2$ and $(x_1,\ldots,x_m,y)\in S_m\cap\mathbb{Z}^{m+1}$, then $R_I=0$. The reason is as follows.

If $|I|=m-1$, then $|I^c|=1$. So we may deduce that $P_{I^c}^{\geq 2}$ is an empty set. It implies that $R_I=0$.

If $|I|=m$, then $I=\{1,2,\ldots,m\}$. So by (\ref{s-R_I}), we know that $R_I$ should be understood as
\begin{equation*}
\Big[q^{y-\sum_{i=1}^m(x_i+1)}\Big](-1)^{m-1}y^{m-2}\frac{e^{-\big(y-\sum_{i=1}^m(x_i+1)\big)\mathcal{W}(q)}}{(1-\mathcal{W}(q))^m}.
\end{equation*}

Since $(x_1,\ldots,x_m,y)\in S_m$, we have $y-\sum_{i=1}^m x_i\leq 1$ by definition. So
\[y-\sum_{i=1}^m(x_i+1)\leq 1-m\leq -1.\]
Here, we recall that we always assume that $m\geq 2$.

Obviously, there will be no terms in
\[(-1)^{m-1}y^{m-2}\frac{e^{-\big(y-\sum_{i=1}^m(x_i+1)\big)\mathcal{W}(q)}}{\bigr(1-\mathcal{W}(q)\bigr)^m}\]
with negative power of $q$. So $R_I=0$.

Now we may reduce (\ref{keyfm}) to
\begin{equation}\label{keyfm2}
F(x_1,\ldots,x_m,y)=\sum_{I\subset\{1,2,\ldots,m\}\atop |I|\leq m-2}R_I,~~~(x_1,\ldots,x_m,y)\in S_m\cap\mathbb{Z}^{m+1}.
\end{equation}

In order to prove Theorems \ref{keythm1} and \ref{keythm2}, we need to show that
\begin{lemma}\label{keylem}
If $|I|\leq m-2$ and $(x_1,\ldots,x_m,y)$ varies in the region
\[D_I:=\Big\{(x_1,\ldots,x_m,y)\in S_m\cap \mathbb{Z}^{m+1}\bigr|\sum_{i\in I}x_i\leq y\Big\},\]
then $R_I$ can be expressed as a polynomial $\mathcal{R}_I$ of $x_1,\ldots,x_m,y$ with degree at most $2m-4$.
\end{lemma}
\begin{proof}
(I) Firstly, we consider those $(x_1,\ldots,x_m,y)\in D_I$ which satisfy $y\geq \sum_{i\in I}(x_i+1)$. Choosing a partition $\{\mathcal{I}_1,\ldots \mathcal{I}_s\}\in P_{I^c}^{\geq 2}$, we only need to show that
\[
\begin{aligned}
\Bigg[\prod_{j\in I^c}w_j^{x_j}q^{y-\sum_{i\in I}(x_i+1)}\Bigg]&\prod_{j\in I^c}(w_j+1)_{x_j}(-1)^{s+|I|-1}y^{s+|I|-2}\times\\
&\frac{e^{-\bigr(y-\sum_{i\in I}(x_i+1)\bigr)\mathcal{W}}}{(1-\mathcal{W})^{|I|}}\prod_{r=1}^s\mathcal{O}_{
\sum_{j\in \mathcal{I}_r}w_j}^{|\mathcal{I}_r|-2}\Bigg(\frac{\mathcal{W}}{(1-\mathcal{W})^3}\Bigg).
\end{aligned}
\]
can be expressed as a polynomial of $x_1,\ldots,x_m,y$ with degree at most $2m-4$.

By (\ref{stir}), we have
\[(w_j+1)_{x_j}=\sum_{b_j=0}^{x_j}A_{x_j}^{b_j}w_j^{b_j}.\]
As for
\[\mathcal{O}_{
\sum_{j\in \mathcal{I}_r}w_j}^{|\mathcal{I}_r|-2}\Bigg(\frac{\mathcal{W}}{(1-\mathcal{W})^3}\Bigg),\]
it becomes
\[
\sum_{a_j\geq 0,j\in\mathcal{I}_r\atop \sum_{j\in \mathcal{I}_r}a_j\leq |\mathcal{I}_r|-2}\frac{(|\mathcal{I}_r|-2)!\prod_{j\in\mathcal{I}_r}w_j^{a_j}}{(|\mathcal{I}_r|-2-\sum_{j\in \mathcal{I}_r}a_j)!\prod_{j\in\mathcal{I}_r}a_j!}Z_{|\mathcal{I}_r|-2-\sum_{j\in\mathcal{I}_r}a_j}
\]
where
\[Z_{|\mathcal{I}_r|-2-\sum_{j\in\mathcal{I}_r}a_j}=\Bigg(\frac{\mathcal{W}}{1-\mathcal{W}}\frac{d}{d\mathcal{W}}\Bigg)^{|\mathcal{I}_r|-2-\sum_{j\in\mathcal{I}_r}a_j}\Bigg(\frac{\mathcal{W}}{(1-\mathcal{W})^3}\Bigg).\]
So after expansion,
\begin{equation}\label{tm}
\begin{aligned}
&\prod_{j\in I^c}(w_j+1)_{x_j}(-1)^{s+|I|-1}y^{s+|I|-2}\times\\
&\frac{e^{-\bigr(y-\sum_{i\in I}(x_i+1)\bigr)\mathcal{W}}}{(1-\mathcal{W})^{|I|}}\prod_{r=1}^s\mathcal{O}_{
\sum_{j\in \mathcal{I}_r}w_j}^{|\mathcal{I}_r|-2}\Bigg(\frac{\mathcal{W}}{(1-\mathcal{W})^3}\Bigg)
\end{aligned}
\end{equation}
becomes
\[
\begin{aligned}
(-1)^{s+|I|-1}y^{s+|I|-2}\frac{e^{-\bigr(y-\sum_{i\in I}(x_i+1)\bigr)\mathcal{W}}}{(1-\mathcal{W})^{|I|}}\sum_{a_j\geq 0,0\leq b_j\leq x_j,j\in I^c\atop \sum_{j\in \mathcal{I}_r}a_j\leq |\mathcal{I}_r|-2,1\leq r\leq s}\prod_{j\in I^c}A_{x_j}^{b_j}w_j^{a_j+b_j}\\
\times \prod_{r=1}^s\frac{(|\mathcal{I}_r|-2)!}{(|\mathcal{I}_r|-2-\sum_{j\in \mathcal{I}_r}a_j)!\prod_{j\in\mathcal{I}_r}a_j!}Z_{|\mathcal{I}_r|-2-\sum_{j\in\mathcal{I}_r}a_j}.
\end{aligned}
\]
Since we want to take the coefficient of the term $\prod_{j\in I^c}w_j^{x_j}q^{y-\sum_{i\in I}(x_i+1)}$, we only need to consider
\[
\begin{aligned}
(-1)^{s+|I|-1}y^{s+|I|-2}\frac{e^{-\bigr(y-\sum_{i\in I}(x_i+1)\bigr)\mathcal{W}}}{(1-\mathcal{W})^{|I|}}\sum_{0\leq a_j\leq x_j,j\in I^c\atop \sum_{j\in \mathcal{I}_r}a_j\leq |\mathcal{I}_r|-2,1\leq r\leq s}\prod_{j\in I^c}A_{x_j}^{x_j-a_j}\\
\times \prod_{r=1}^s\frac{(|\mathcal{I}_r|-2)!}{(|\mathcal{I}_r|-2-\sum_{j\in \mathcal{I}_r}a_j)!\prod_{j\in\mathcal{I}_r}a_j!}Z_{|\mathcal{I}_r|-2-\sum_{j\in\mathcal{I}_r}a_j}
\end{aligned}
\]
and then take the coefficient of the term $q^{y-\sum_{i\in I}(x_i+1)}$. If we set $A_{x_j}^{x_j-a_j}=0$ when $x_j<a_j$, then we may remove the constraint that $a_j\leq x_j$. So it becomes
\[\begin{aligned}
(-1)^{s+|I|-1}y^{s+|I|-2}\frac{e^{-\bigr(y-\sum_{i\in I}(x_i+1)\bigr)\mathcal{W}}}{(1-\mathcal{W})^{|I|}}\sum_{a_j\geq 0, j\in I^c\atop \sum_{j\in \mathcal{I}_r}a_j\leq |\mathcal{I}_r|-2,1\leq r\leq s}\prod_{j\in I^c}A_{x_j}^{x_j-a_j}\\
\times \prod_{r=1}^s\frac{(|\mathcal{I}_r|-2)!}{(|\mathcal{I}_r|-2-\sum_{j\in \mathcal{I}_r}a_j)!\prod_{j\in\mathcal{I}_r}a_j!}Z_{|\mathcal{I}_r|-2-\sum_{j\in\mathcal{I}_r}a_j}.
\end{aligned}\]

Now we only need to show that for fixed $a_j\in \mathbb{Z},j\in I^c$ which satisfy
\begin{equation}\label{cons1}
a_j\geq 0,~~~\sum_{j\in \mathcal{I}_r}a_j\leq |\mathcal{I}_r|-2,~1\leq r\leq s,
\end{equation}
term
\[
\prod_{j\in I^c}A_{x_j}^{x_j-a_j}[q^{y-\sum_{i\in I}(x_i+1)}]y^{s+|I|-2}\frac{e^{-\bigr(y-\sum_{i\in I}(x_i+1)\bigr)\mathcal{W}}}{(1-\mathcal{W})^{|I|}}\prod_{r=1}^sZ_{|\mathcal{I}_r|-2-\sum_{j\in\mathcal{I}_r}a_j}
\]
can be expressed as a polynomial of $x_1,\ldots,x_m,y$ with degree at most $2m-4$.

By \cite{CGHJK} Corollary 8.2, we know that if $x_j\geq a_j$, then
\[A_{x_j}^{x_j-a_j}=\sum_{e=0}^{a_j}\sum_{f=0}^e(-1)^{f+a_j}{e\choose f}{x_j+e\choose a_j+e}{x_j+a_j+1\choose a_j-e}\frac{f^{a_j+e}}{e!}.\]
For fixed $a_j\geq 0$, the R.H.S can be written as a polynomial of $x_j$ with degree at most $2a_j$. It is also easy to check that if $x_j< a_j$ and $x_j\in \mathbb{Z}_{\geq0}$, then such polynomial vanishes as well. So we are left to show that
\[[q^{y-\sum_{i\in I}(x_i+1)}]y^{s+|I|-2}\frac{e^{-\bigr(y-\sum_{i\in I}(x_i+1)\bigr)\mathcal{W}}}{(1-\mathcal{W})^{|I|}}\prod_{r=1}^sZ_{|\mathcal{I}_r|-2-\sum_{j\in\mathcal{I}_r}a_j}\]
can be expressed as as a polynomial of $x_1,\ldots,x_m,y$ with degree at most $2m-4-\sum_{j\in I^c}2a_j$. By Lemma \ref{fmrl}, we know that
\[\prod_{r=1}^sZ_{|\mathcal{I}_r|-2-\sum_{j\in\mathcal{I}_r}a_j}=\frac{\mathcal{W}^s\prod_{r=1}^s Q_{|\mathcal{I}_r|-2-\sum_{j\in\mathcal{I}_r}a_j}}{(1-\mathcal{W})^{\sum_{r=1}^s(2|\mathcal{I}_r|-1-2\sum_{j\in \mathcal{I}_r}a_j)}}\]
where $Q_{|\mathcal{I}_r|-2-\sum_{j\in\mathcal{I}_r}a_j}$ is a polynomial of $\mathcal{W}$ with degree at most $|\mathcal{I}_r|-2-\sum_{j\in\mathcal{I}_r}a_j$. So
\[\prod_{r=1}^sZ_{|\mathcal{I}_r|-2-\sum_{j\in\mathcal{I}_r}a_j}\]
can be written as a sum of terms in the following form:
\[C\frac{W^{D}}{(1-\mathcal{W})^{\sum_{r=1}^s(2|\mathcal{I}_r|-1-2\sum_{j\in \mathcal{I}_r}a_j)}}\]
where $C$ is some constant and
\begin{equation}\label{cons2}
s\leq D\leq \sum_{r=1}^s(|\mathcal{I}_r|-1-\sum_{j\in \mathcal{I}_r}a_j).
\end{equation}
So we only need to show that each term
\begin{equation}\label{W-plT}
Cy^{s+|I|-2}[q^{y-\sum_{i\in I}(x_i+1)}]\frac{\mathcal{W}^De^{-\bigr(y-\sum_{i\in I}(x_i+1)\bigr)\mathcal{W}}}{(1-\mathcal{W})^{|I|+\sum_{r=1}^s(2|\mathcal{I}_r|-1-2\sum_{j\in \mathcal{I}_r}a_j)}}
\end{equation}
can be expressed as as a polynomial of $x_1,\ldots,x_m,y$ with degree at most $2m-4-\sum_{j\in I^c}2a_j$.

Since $y-\sum_{i\in I}(x_i+1)\geq 0$, by Lemma \ref{comb} we know that
\[[q^{y-\sum_{i\in I}(x_i+1)}]\frac{\mathcal{W}^De^{-\bigr(y-\sum_{i\in I}(x_i+1)\bigr)\mathcal{W}}}{(1-\mathcal{W})^{|I|+\sum_{r=1}^s(2|\mathcal{I}_r|-1-2\sum_{j\in \mathcal{I}_r}a_j)}}\]
equals to
\begin{equation}\label{W-pl}
{|I|+\sum_{r=1}^s(2|\mathcal{I}_r|-1-2\sum_{j\in \mathcal{I}_r}a_j)-2+y-\sum_{i\in I}(x_i+1)-D\choose |I|+\sum_{r=1}^s(2|\mathcal{I}_r|-1-2\sum_{j\in \mathcal{I}_r}a_j)-2}
\end{equation}
by using the convention of (\ref{conven}).

Since $|I|\leq m-2$, we have $|I^c|\geq 2$. The condition that
\[\bigcup_{r=1}^s\mathcal{I}_r=I^c\neq \emptyset\]
implies that $s\geq 1$. Now by (\ref{cons1}), we have
\[|I|+\sum_{r=1}^s(2|\mathcal{I}_r|-1-2\sum_{j\in \mathcal{I}_r}a_j)-2\geq |I|+3s-2> 0.\]
By (\ref{cons1}) and (\ref{cons2}), we may also deduce that
\[
\begin{aligned}
&|I|+\sum_{r=1}^s(2|\mathcal{I}_r|-1-2\sum_{j\in \mathcal{I}_r}a_j)-2+y-\sum_{i\in I}(x_i+1)-D\\
&\geq|I|+y-\sum_{i\in I}(x_i+1)+2s-2\geq |I|+2s-2\geq 0.
\end{aligned}
\]
So by (\ref{convenT}), we know that (\ref{W-pl}) can be written as
\begin{equation}\label{W-pl-1}
\frac{X(X-1)\ldots(X-(|I|+\sum_{r=1}^s(2|\mathcal{I}_r|-1-2\sum_{j\in \mathcal{I}_r}a_j)-2)+1)}{(|I|+\sum_{r=1}^s(2|\mathcal{I}_r|-1-2\sum_{j\in \mathcal{I}_r}a_j)-2)!}
\end{equation}
where
\[X=|I|+\sum_{r=1}^s(2|\mathcal{I}_r|-1-2\sum_{j\in \mathcal{I}_r}a_j)-2+y-\sum_{i\in I}(x_i+1)-D.\]
It is a polynomial of $x_1,\ldots,x_m,y$ with degree at most
\[|I|+\sum_{r=1}^s(2|\mathcal{I}_r|-1-2\sum_{j\in \mathcal{I}_r}a_j)-2.\]

(i) Now if $s+|I|-2\geq 0$, then it is easy to see that (\ref{W-plT}) is a polynomial of $x_1,\ldots,x_m,y$ with degree at most
\[2|I|+\sum_{r=1}^s(2|\mathcal{I}_r|-2\sum_{j\in \mathcal{I}_r}a_j)-4=2m-4-2\sum_{j\in I^c}a_j\]
where we have used the fact that
\[|I|+\sum_{r=1}^s|\mathcal{I}_r|=|I|+|I^c|=m.\]

(ii) If $s+|I|-2<0$, then it forces that $s=1$ and $|I|=0$ by the fact that $s\geq 1$. Then $\mathcal{I}_1=\{1,2,\ldots,m\}$ and $s+|I|-2=-1$. Now (\ref{W-plT}) becomes
\[\frac{C}{y}[q^y]\frac{\mathcal{W}^{D}e^{-y\mathcal{W}}}{(1-\mathcal{W})^{2m-1-2\sum_{j=1}^m a_j}}.\]
It equals to
\begin{equation}\label{s-W-plT}
\frac{C}{y}\frac{Y(Y-1)\ldots(Y-(2m-3-2\sum_{j=1}^ma_j)+1)}{(2m-3-2\sum_{j=1}^ma_j)!}
\end{equation}
by Lemma \ref{comb}, where
\[Y=2m-3-2\sum_{j=1}^ma_j+y-D.\]
By (\ref{cons1}) and (\ref{cons2}), we have
\[Y=2m-2\sum_{j=1}^ma_j-3+y-D\geq m-2-\sum_{j=1}^ma_j+y\geq y.\]
And by (\ref{cons2}), we have
\[Y-(2m-3-2\sum_{j=1}^ma_j)+1=y-D+1\leq y.\]
So there must exist one factor in $Y(Y-1)\ldots(Y-(2m-3-2\sum_{j=1}^ma_j)+1)$ which is exactly $y$. It further implies that (\ref{s-W-plT}) is still a polynomial of $x_1,\ldots,x_m,y$ with degree at most $2m-4-2\sum_{j\in I^c}a_j$.

We have proven that there exists one polynomial $\mathcal{R}_I$ of $x_1,\ldots,x_m,y$ with degree at most $2m-4$ such that
\begin{equation}\label{plyty}
R_I=\mathcal{R}_I
\end{equation}
when $y\geq \sum_{i\in I}(x_i+1)$ and $(x_1,\ldots,x_m,y)\in S_m\cap \mathbb{Z}^{m+1}$.

(II) In order to prove the lemma, we are left to show that (\ref{plyty}) still holds when
\begin{equation}\label{cons3}
\sum_{i\in I}x_i\leq y< \sum_{i\in I}(x_i+1),~~~(x_1,\ldots,x_m,y)\in S_m\cap \mathbb{Z}^{m+1}.
\end{equation}
We recall that $R_I$ is given by the coefficient of term $\prod_{j\in I^c}w_j^{x_j}q^{y-\sum_{i\in I}(x_i+1)}$ in (\ref{tm}). Since
$y-\sum_{i\in I}(x_i+1)<0$ and there will be no terms in (\ref{tm}) with negative power of $q$, we have $R_I=0$. Then we only need to show that in this case $\mathcal{R}_I=0$ as well. We will prove it by showing that there always exists one factor in the numerator of (\ref{W-pl-1}) which is $0$.

Since the first factor in the numerator of (\ref{W-pl-1}) is
\begin{eqnarray*}
X&=&|I|+\sum_{r=1}^s(2|\mathcal{I}_r|-1-2\sum_{j\in \mathcal{I}_r}a_j)-2+y-\sum_{i\in I}(x_i+1)-D\\
 &=&\sum_{r=1}^s(2|\mathcal{I}_r|-1-2\sum_{j\in \mathcal{I}_r}a_j)-D-2+y-\sum_{i\in I}x_i\\
 &\stackrel{(\ref{cons2})(\ref{cons3})}{\geq} & \sum_{r=1}^s(|\mathcal{I}_r|-\sum_{j\in \mathcal{I}_r}a_j)-2\\
 &\stackrel{(\ref{cons1})}{\geq}& 2s-2\geq 0,
\end{eqnarray*}
the last factor in the numerator of (\ref{W-pl-1}) is
\begin{eqnarray*}
& &X-(|I|+\sum_{r=1}^s(2|\mathcal{I}_r|-1-2\sum_{j\in \mathcal{I}_r}a_j)-2)+1\\
&=&y-\sum_{i\in I}(x_i+1)-D+1\\
&\stackrel{(\ref{cons2})(\ref{cons3})}{<} & -s+1\leq 0,\\
\end{eqnarray*}
and $X,X-1,\ldots,X-(|I|+\sum_{r=1}^s(2|\mathcal{I}_r|-1-2\sum_{j\in \mathcal{I}_r}a_j)-2)+1$ are all integers, we may deduce that there must exist
one factor in the numerator of (\ref{W-pl-1}) which is $0$.

The proof of Lemma \ref{keylem} is complete.
\end{proof}

Using Lemma \ref{keylem}, Theorems \ref{keythm1} and \ref{keythm2} can be proved as follows.

\begin{proof}[Proof of Theorems \ref{keythm1} and \ref{keythm2}]Let $\mathfrak{c}$ be one chamber of $S_m$. Then for each $I\subset \{1,2,\ldots,m\}$ such that $|I|\leq m-2$, either
\begin{equation}\label{chos1}
\sum_{i\in I}x_i<y,~~~\forall (x_1,\ldots,x_m,y)\in \mathfrak{c}
\end{equation}
or
\begin{equation}\label{chos2}
\sum_{i\in I}x_i>y,~~~\forall (x_1,\ldots,x_m,y)\in \mathfrak{c}
\end{equation}
We use $\mathcal{S}_{\mathfrak{c}}^1$ (resp. $\mathcal{S}_{\mathfrak{c}}^2$) to denote the set of those $I$ which satisfy (\ref{chos1}) (resp. (\ref{chos2})) with $|I|\leq m-2$. Now we claim that
\begin{equation}\label{claim}
P_{\mathfrak{c}}(x_1,\ldots,x_m,y)=\sum_{I\in \mathcal{S}_{\mathfrak{c}}^1}R_{I},~~~(x_1,\ldots,x_m,y)\in \overline{\mathfrak{c}}\cap \mathbb{Z}^{m+1}.
\end{equation}
By (\ref{keyfm2}), we know that
\begin{equation*}
P_{\mathfrak{c}}(x_1,\ldots,x_m,y)=\sum_{I\subset\{1,2,\ldots,m\}\atop |I|\leq m-2}R_I,~~~(x_1,\ldots,x_m,y)\in \overline{\mathfrak{c}}\cap \mathbb{Z}^{m+1}
\end{equation*}
where $\overline{\mathfrak{c}}$ is the closure of $\mathfrak{c}$ in $S_m$. So we only need to show that for each $I\in \mathcal{S}_{\mathfrak{c}}^2$,
\[R_I=0,~~~\forall (x_1,\ldots,x_m,y)\in \overline{\mathfrak{c}}\cap \mathbb{Z}^{m+1}.\]
Recall that $R_I$ is given by
\[
\begin{aligned}
\Bigg[\prod_{j\in I^c}w_j^{x_j}q^{y-\sum_{i\in I}(x_i+1)}\Bigg]&\prod_{j\in I^c}(w_j+1)_{x_j}(-1)^{s+|I|-1}y^{s+|I|-2}\times\\
&\frac{e^{-\bigr(y-\sum_{i\in I}(x_i+1)\bigr)\mathcal{W}}}{(1-\mathcal{W})^{|I|}}\prod_{r=1}^s\mathcal{O}_{
\sum_{j\in \mathcal{I}_r}w_j}^{|\mathcal{I}_r|-2}\Bigg(\frac{\mathcal{W}}{(1-\mathcal{W})^3}\Bigg).
\end{aligned}
\]
By the definition of $\mathcal{S}_{\mathfrak{c}}^2$, it is easy to see that for each $I\in \mathcal{S}_{\mathfrak{c}}^2$,
\begin{equation*}
\sum_{i\in I}x_i\geq y,~~~\forall (x_1,\ldots,x_m,y)\in \overline{\mathfrak{c}}.
\end{equation*}
Since $y\in S_m$, we know that $y\geq 1$. $I$ can not be empty then. Otherwise,
\[\sum_{i\in I}x_i=0\geq y\geq 1\]
which is a contradiction. So we have
\[y-\sum_{i\in I}(x_i+1)=y-\sum_{i\in I}x_i-|I|\leq -|I|<0.\]
Obviously, there are no terms in
\[
\begin{aligned}
&\prod_{j\in I^c}(w_j+1)_{x_j}(-1)^{s+|I|-1}y^{s+|I|-2}\times\\
&\frac{e^{-\bigr(y-\sum_{i\in I}(x_i+1)\bigr)\mathcal{W}}}{(1-\mathcal{W})^{|I|}}\prod_{r=1}^s\mathcal{O}_{
\sum_{j\in \mathcal{I}_r}w_j}^{|\mathcal{I}_r|-2}\Bigg(\frac{\mathcal{W}}{(1-\mathcal{W})^3}\Bigg)
\end{aligned}
\]
with negative power of $q$. So $R_I=0$ for each $I\in \mathcal{S}_{\mathfrak{c}}^2$.

We have proven claim (\ref{claim}). Next, we note that
for each $I\in \mathcal{S}_{\mathfrak{c}}^1$,
\begin{equation*}
\sum_{i\in I}x_i\leq y,~~~\forall (x_1,\ldots,x_m,y)\in \overline{\mathfrak{c}}.
\end{equation*}
So by Lemma \ref{keylem} and claim (\ref{claim}), it is easy to see that
\begin{equation}\label{keyfm3}
P_{\mathfrak{c}}(x_1,\ldots,x_m,y)=\sum_{I\in \mathcal{S}_{\mathfrak{c}}^1}\mathcal{R}_{I},~~~(x_1,\ldots,x_m,y)\in \overline{\mathfrak{c}}\cap \mathbb{Z}^{m+1}
\end{equation}
where $\mathcal{R}_{I}$ is a polynomial of $x_1,\ldots,x_m,y$ with degree at most $2m-4$.

Now Theorems \ref{keythm1}, \ref{keythm2} directly follows (\ref{keyfm3}).
\end{proof}
In order to prove Theorem \ref{keythm3}, we recall that the totally negative chamber $\mathfrak{c}^{tn}$ is a chamber such that $\sum_{i\in I}x_i<y$ for all $I\subset \{1,2,\ldots,m\}$ with $|I|\leq m-2$. And we need to show that
\begin{equation*}
P_{\mathfrak{c}^{tn}}=y^{m-2}{1+\sum_{i=1}^m x_i-y+m-2\choose m-2}.
\end{equation*}
\begin{proof}[Proof of Theorem \ref{keythm3}]
By (\ref{keyfm3}), we know that
\begin{equation*}
P_{\mathfrak{c}^{tn}}(x_1,\ldots,x_m,y)=\sum_{I\in \mathcal{S}_{\mathfrak{c}^{tn}}^1}\mathcal{R}_{I},~~~(x_1,\ldots,x_m,y)\in \overline{\mathfrak{c}^{tn}}\cap \mathbb{Z}^{m+1}.
\end{equation*}
In this case, $\mathcal{S}_{\mathfrak{c}^{tn}}^1$ consists of all of those $I\subset \{1,2,\ldots,m\}$ such that $|I|\leq m-2$. So we also have
\begin{equation}\label{latm}
P_{\mathfrak{c}^{tn}}(x_1,\ldots,x_m,y)=\sum_{I\subset \{1,\ldots,m\}\atop |I|\leq m-2}\mathcal{R}_{I},~~~(x_1,\ldots,x_m,y)\in \overline{\mathfrak{c}^{tn}}\cap \mathbb{Z}^{m+1}.
\end{equation}
We define
\[D_m:=\Big\{(x_1,\ldots,x_m,y)\in \mathbb{Z}^{m+1}\Big|x_1\geq \ldots\geq x_m\geq 0,y\geq \sum_{i=1}^m(x_i+1)\Big\}.\]
Obviously, $D_m\cap S_m=\emptyset$. But by (\ref{contri}), (\ref{crstm}), (\ref{imptm}) and (\ref{s-imptm}), we can also deduce that
\begin{equation}\label{new1}
\Big\langle y\Big |\prod_{i=1}^m \prod_{k_i=0}^{x_i}(k_i\tau+[\infty])\Big\rangle_0^{T}=t^{1+\sum_{i=1}^m x_i-y}\sum_{I\subset\{1,2,\ldots,m\}}R_I,
\end{equation}
where $(x_1,\ldots,x_m,y)\in D_m$. It is easy to see that
\[\sum_{i=1}^m x_i+1-y<0,~~~ \forall (x_1,\ldots,x_m,y)\in D_m.\]
Since the L.H.S of (\ref{new1}) does not contain terms with negative power of $t$, we may deduce that
\[\Big\langle y\Big |\prod_{i=1}^m \prod_{k_i=0}^{x_i}(k_i\tau+[\infty])\Big\rangle_0^{T}=0.\]
It further implies that
\[\sum_{I\subset\{1,2,\ldots,m\}}R_I=\sum_{I\subset\{1,2,\ldots,m\}\atop |I|\leq m-2}R_I+\sum_{I\subset\{1,2,\ldots,m\}\atop |I|> m-2}R_I=0,~\forall (x_1,\ldots,x_m,y)\in D_m.\]
The part
\[\sum_{I\subset\{1,2,\ldots,m\}\atop |I|\leq m-2}R_I=\sum_{I\subset\{1,2,\ldots,m\}\atop |I|\leq m-2}\mathcal{R}_I,~~~\forall(x_1,\ldots,x_m,y)\in D_m\]
by Lemma \ref{keylem}. As for the part
\[\sum_{I\subset\{1,2,\ldots,m\}\atop |I|> m-2}R_I,\]
we claim that it equals to
\[(-1)^{m-1}y^{m-2}{y-\sum_{i=1}^m(x_i+1)+m-2\choose m-2},~~~\forall(x_1,\ldots,x_m,y)\in D_m.\]
The reason is as follows.

If $|I|=m-1$, then $|I^c|=1$. So we may deduce that $P_{I^c}^{\geq 2}$ is an empty set. It implies that $R_I=0$.

If $|I|=m$, then $I=\{1,2,\ldots,m\}$. So by (\ref{s-imptm}), we know that $R_I$ is given by
\begin{equation*}
\Big[q^{y-\sum_{i=1}^m(x_i+1)}\Big](-1)^{m-1}y^{m-2}\frac{e^{-\big(y-\sum_{i=1}^m(x_i+1)\big)\mathcal{W}(q)}}{(1-\mathcal{W}(q))^m}.
\end{equation*}

The latter equals to
\[(-1)^{m-1}y^{m-2}{y-\sum_{i=1}^m(x_i+1)+m-2\choose m-2}\]
by Lemma \ref{comb}.

Now we have shown that
\[\sum_{I\subset\{1,2,\ldots,m\}\atop |I|\leq m-2}\mathcal{R}_I+(-1)^{m-1}y^{m-2}{y-\sum_{i=1}^m(x_i+1)+m-2\choose m-2}=0\]
for $\forall(x_1,\ldots,x_m,y)\in D_m$. So as a polynomial of $x_1,\ldots,x_m,y$,
\[\sum_{I\subset\{1,2,\ldots,m\}\atop |I|\leq m-2}\mathcal{R}_I=(-1)^{m-2}y^{m-2}{y-\sum_{i=1}^m(x_i+1)+m-2\choose m-2}.\]
Here, as a polynomial of $x_1,\ldots,x_m,y$, the R.H.S equals to $1$ if $m=2$, and
\[(-1)^{m-2}y^{m-2}\frac{\prod_{j=0}^{m-3}(y-\sum_{i=1}^m(x_i+1)+m-2-j)}{(m-2)!}\]
if $m\geq 3$. The latter can also be written as
\[y^{m-2}\frac{\prod_{j=0}^{m-3}(1+\sum_{i=1}^mx_i-y+m-2-j)}{(m-2)!}.\]
So as a polynomial,
\begin{equation}\label{latm2}
P_{\mathfrak{c}^{tn}}=y^{m-2}\frac{\prod_{j=0}^{m-3}(1+\sum_{i=1}^mx_i-y+m-2-j)}{(m-2)!}
\end{equation}
by (\ref{latm}). If $(x_1,\ldots,x_m,y)$ varies in $\overline{\mathfrak{c}^{tn}}\cap \mathbb{Z}^{m+1}$, then $1+\sum_{i=1}^mx_i-y\geq 0$.
So the R.H.S of (\ref{latm2}) can be further written as
\[y^{m-2}{1+\sum_{i=1}^m x_i-y+m-2\choose m-2}.\]
Theorem \ref{keythm3} directly follows.
\end{proof}

\section{Examples}\label{3}
We will give examples of chamber structure when $m=2,3$.

\subsection{$m=2$}
There will be only one chamber in the parameter space $S_2$. So we have
\[\begin{array}{rl}
  \text{Chamber:} &  \mathfrak{c}^{tn}=S_2.\\
                        \\
  \text{Polynomial:} & P_{\mathfrak{c}^{tn}}=1. \\
\end{array}\]

\subsection{$m=3$}

There are four chambers in the parameter space $S_3$. So we have

\[
\begin{array}{rl}
  \text{Chambers:} & \mathfrak{c}^{tn}=\{y>x_1\}\cap S_3, \\
                   & \mathfrak{c}_1=\{x_1>y>x_2\}\cap S_3,\\
                   & \mathfrak{c}_2=\{x_2>y>x_3\}\cap S_3,\\
                   & \mathfrak{c}_3=\{x_3>y\}\cap S_3.\\
                   \\
  \text{Polynomials:} & P_{\mathfrak{c}^{tn}}=y(\sum_{i=1}^3x_i+2-y), \\
                      & P_{\mathfrak{c}_1}=y(\sum_{i=1}^3x_i+2-y)+\frac{(y-x_1)(y-x_1-1)}{2}, \\
                      & P_{\mathfrak{c}_2}=y(\sum_{i=1}^3x_i+2-y)+\sum_{i=1}^2\frac{(y-x_i)(y-x_i-1)}{2},\\
                      & P_{\mathfrak{c}_3}=y(\sum_{i=1}^3x_i+2-y)+\sum_{i=1}^3\frac{(y-x_i)(y-x_i-1)}{2}.\\
\end{array}
\]

\section{Appendix}

The following two lemmas play an important role in the proof of Lemma \ref{keylem}.

Recall that
\[\mathcal{W}(q):=\sum_{n=1}^{\infty}\frac{n^{n-1}}{n!}q^n\]
which satisfies $q=\mathcal{W}(q)e^{-\mathcal{W}(q)}$ by (\ref{invsrls}).

\begin{lemma}\label{fmrl}
For $k\in\mathbb{Z}_{\geq 0}$, we may express
\[\Bigg(\frac{\mathcal{W}}{1-\mathcal{W}}\frac{d}{d\mathcal{W}}\Bigg)^k\Bigg(\frac{\mathcal{W}}{(1-\mathcal{W})^3}\Bigg)\]
as
\[\frac{\mathcal{W}Q_k(\mathcal{W})}{(1-\mathcal{W})^{2k+3}}\]
where $Q_k(\mathcal{W})$ is a polynomial of $\mathcal{W}$ with degree at most $k$.
\end{lemma}
\begin{proof}
We will prove it by induction. The case $k=0$ is obviously. Now we suppose that for $k=m$,
\[\Bigg(\frac{\mathcal{W}}{1-\mathcal{W}}\frac{d}{d\mathcal{W}}\Bigg)^m\Bigg(\frac{\mathcal{W}}{(1-\mathcal{W})^3}\Bigg)=\frac{\mathcal{W}Q_m(\mathcal{W})}{(1-\mathcal{W})^{2m+3}}\]
such that $Q_m(\mathcal{W})$ is a polynomial of $\mathcal{W}$ with degree at most $m$. Then we have
\[\Bigg(\frac{\mathcal{W}}{1-\mathcal{W}}\frac{d}{d\mathcal{W}}\Bigg)^{m+1}\Bigg(\frac{\mathcal{W}}{(1-\mathcal{W})^3}\Bigg)=\frac{\mathcal{W}Q_{m+1}(\mathcal{W})}{(1-\mathcal{W})^{2m+5}}.\]
Here,
\[Q_{m+1}(\mathcal{W})=(1-\mathcal{W})Q_m'(\mathcal{W})+(2m+3)Q_m(\mathcal{W}).\]
Obviously, $Q_{m+1}(\mathcal{W})$ is a polynomial of $\mathcal{W}$ with degree at most $m+1$.
\end{proof}

Let $\alpha,\beta\in \mathbb{Z}$. We will use the following convention
\begin{equation}\label{conven}
{\alpha\choose \beta}=
\begin{cases}
  0, & \mbox{if } \alpha<\beta, \\
  1, & \mbox{if } \alpha=\beta, \\
  0, & \mbox{if } \alpha>\beta,~\beta<0,\\
  \frac{(\alpha)!}{\beta!(\alpha-\beta)!}, & \mbox{if } \alpha>\beta,~\beta\geq 0.
\end{cases}
\end{equation}
Obviously, if $\alpha,\beta\in \mathbb{Z}_{\geq 0}$, then
\begin{equation}\label{convenT}
{\alpha\choose \beta}=
\begin{cases}
1, & \mbox{if } \beta=0,\\
\frac{\alpha(\alpha-1)\ldots (\alpha-\beta+1)}{\beta!}, & \mbox{if } \beta>0.
\end{cases}
\end{equation}

Let $a,\mu\in\mathbb{Z}_{\geq 0}$ and $b\in \mathbb{Z}_{>0}$. Using the above notation, we have
\begin{lemma}\label{comb}
After expanding
\[\frac{\mathcal{W}^ae^{-\mu \mathcal{W}}}{(1-\mathcal{W})^b}\]
according to $q$, the coefficient of term $q^{\mu}$ is given by
\[{b-2+\mu-a\choose b-2}.\]
\end{lemma}
\begin{proof}
It is easy to see that
\[\frac{\mathcal{W}^{a}e^{-\mu \mathcal{W}}}{(1-\mathcal{W})^{b}}=(\mathcal{W}e^{-\mathcal{W}})^{a}\frac{e^{-(\mu-a)\mathcal{W}}}{(1-\mathcal{W})^{b}}=q^a\frac{e^{-(\mu-a)\mathcal{W}}}{(1-\mathcal{W})^{b}}.\]
So the coefficient of the term $q^{\mu}$ in
\[\frac{\mathcal{W}^{a}e^{-\mu \mathcal{W}}}{(1-\mathcal{W})^{b}}\]
is just the coefficient of the term $q^{\mu-a}$ in
\[\frac{e^{-(\mu-a)\mathcal{W}}}{(1-\mathcal{W})^{b}}.\]
So we only need to show that the coefficient of the term $q^{\mu}$ in
\[\frac{e^{-\mu \mathcal{W}}}{(1-\mathcal{W})^{n}}\]
is given by
\[{n-2+\mu\choose n-2}.\]

In order to prove this claim, we firstly show that
\[\frac{e^{\mathcal{W}x}}{(1-\mathcal{W})^n}=\sum_{k=0}^{\infty}\Bigg(\sum_{m=0}^k{k-m+n-2\choose n-2}\frac{(k+x)^m}{m!}\Bigg)q^k.\]

We will prove it by induction. Obviously, the case $n=1$ directly follows from (\ref{keyid}). Suppose that it is true for $n=r$, then we have
\[\int_{x}^{+\infty}\frac{e^{\mathcal{W}y-y}}{(1-\mathcal{W})^r}dy=\sum_{k=0}^{\infty}\Bigg(\sum_{m=0}^k{k-m+r-2\choose r-2}\int_{x}^{+\infty}\frac{(k+y)^m}{m!}e^{-y}dy\Bigg)q^k.\]

The L.H.S equals to
\[\frac{e^{\mathcal{W}x-x}}{(1-\mathcal{W})^{r+1}}.\]
As for the R.H.S, via integration by parts it is easy to compute that
\[\int_{x}^{+\infty}\frac{(k+y)^m}{m!}e^{-y}dy=e^{-x}\Bigg(\sum_{l=0}^m \frac{(k+x)^l}{l!}\Bigg).\]
So the part
\[
\begin{aligned}
 &\sum_{m=0}^k{k-m+r-2\choose r-2}\int_{x}^{+\infty}\frac{(k+y)^m}{m!}e^{-y}dy=\\
&e^{-x}\sum_{l=0}^k\frac{(k+x)^l}{l!}\sum_{m=l}^k{k-m+r-2\choose r-2}=e^{-x}\sum_{l=0}^k\frac{(k+x)^l}{l!}{k-l+r-1\choose r-1}.
\end{aligned}
\]
Here, we have used the combinatorial identity
\[\sum_{m=l}^k{k-m+r-2\choose r-2}={k-l+r-1\choose r-1}.\]

So we get the equality for $n=r+1$
\[\frac{e^{\mathcal{W}x}}{(1-\mathcal{W})^{r+1}}=\sum_{k=0}^{\infty}\Bigg(\sum_{l=0}^k{k-l+r-1\choose r-1}\frac{(k+x)^l}{l!}\Bigg)q^k.\]

Now it is easy to compute that the coefficient of the term $q^{\mu}$ in
\[\frac{e^{-\mu \mathcal{W}}}{(1-\mathcal{W})^{n}}\]
equals to
\[{n-2+\mu\choose n-2}.\]
\end{proof}

\bibliographystyle{plain}

\begin{thebibliography}{99}

\bibitem{CJM}
Renzo Cavalieri, Paul Johnson, and Hannah Markwig.
\newblock Wall crossings for double {H}urwitz numbers.
\newblock {\em Adv. Math.}, 228(4):1894--1937, 2011.

\bibitem{C}
Charalambos~A. Charalambides.
\newblock {\em Enumerative combinatorics}.
\newblock CRC Press Series on Discrete Mathematics and its Applications.
  Chapman \& Hall/CRC, Boca Raton, FL, 2002.

\bibitem{CGHJK}
R.~M. Corless, G.~H. Gonnet, D.~E.~G. Hare, D.~J. Jeffrey, and D.~E. Knuth.
\newblock On the {L}ambert {$W$} function.
\newblock {\em Adv. Comput. Math.}, 5(4):329--359, 1996.

\bibitem{FP}
C.~Faber and R.~Pandharipande.
\newblock Relative maps and tautological classes.
\newblock {\em J. Eur. Math. Soc. (JEMS)}, 7(1):13--49, 2005.

\bibitem{Ga}
A.~Gathmann.
\newblock {G}romov-{W}itten invariants of hypersurfaces.
\newblock Habilitation thesis, Univ. of Kaiserslautern, 2003.

\bibitem{GJV}
I.~P. Goulden, D.~M. Jackson, and R.~Vakil.
\newblock Towards the geometry of double {H}urwitz numbers.
\newblock {\em Adv. Math.}, 198(1):43--92, 2005.

\bibitem{GV}
Tom Graber and Ravi Vakil.
\newblock Relative virtual localization and vanishing of tautological classes
  on moduli spaces of curves.
\newblock {\em Duke Math. J.}, 130(1):1--37, 2005.

\bibitem{IP}
Eleny-Nicoleta Ionel and Thomas~H. Parker.
\newblock Gromov-{W}itten invariants of symplectic sums.
\newblock {\em Math. Res. Lett.}, 5(5):563--576, 1998.

\bibitem{J}
Paul Johnson.
\newblock Double {H}urwitz numbers via the infinite wedge.
\newblock {\em Trans. Amer. Math. Soc.}, 367(9):6415--6440, 2015.

\bibitem{LR}
An-Min Li and Yongbin Ruan.
\newblock Symplectic surgery and {G}romov-{W}itten invariants of {C}alabi-{Y}au
  3-folds.
\newblock {\em Invent. Math.}, 145(1):151--218, 2001.

\bibitem{Li}
Jun Li.
\newblock Stable morphisms to singular schemes and relative stable morphisms.
\newblock {\em J. Differential Geom.}, 57(3):509--578, 2001.

\bibitem{LLZ2}
Chiu-Chu~Melissa Liu, Kefeng Liu, and Jian Zhou.
\newblock A proof of a conjecture of {M}ari\~no-{V}afa on {H}odge integrals.
\newblock {\em J. Differential Geom.}, 65(2):289--340, 2003.

\bibitem{LLZ1}
Chiu-Chu~Melissa Liu, Kefeng Liu, and Jian Zhou.
\newblock A formula of two-partition {H}odge integrals.
\newblock {\em J. Amer. Math. Soc.}, 20(1):149--184, 2007.

\bibitem{MP}
D.~Maulik and R.~Pandharipande.
\newblock A topological view of {G}romov-{W}itten theory.
\newblock {\em Topology}, 45(5):887--918, 2006.

\bibitem{OP1}
A.~Okounkov and R.~Pandharipande.
\newblock Gromov-{W}itten theory, {H}urwitz theory, and completed cycles.
\newblock {\em Ann. of Math. (2)}, 163(2):517--560, 2006.

\bibitem{OP3}
A.~Okounkov and R.~Pandharipande.
\newblock Virasoro constraints for target curves.
\newblock {\em Invent. Math.}, 163(1):47--108, 2006.

\bibitem{SSV}
S.~Shadrin, M.~Shapiro, and A.~Vainshtein.
\newblock Chamber behavior of double {H}urwitz numbers in genus 0.
\newblock {\em Adv. Math.}, 217(1):79--96, 2008.

\bibitem{W}
Longting Wu.
\newblock A remark on {G}romov-{W}itten invariants of quintic threefold.
\newblock In preparation

\end{thebibliography}

\end{document}